\documentclass[11pt,oneside]{amsart}

\usepackage{amsmath,txfonts}
\usepackage[dvips]{graphicx}

\theoremstyle{plain} 
\newtheorem{thm}{Theorem}[section]
\newtheorem{pro}{Proposition}[section]
\newtheorem{exa}{Example}[section]
\newtheorem{rem}{Remark}[section]
\newtheorem{lem}{Lemma}[section]

\newtheorem{cor}{Corollary}[section]

\numberwithin{thm}{section}
\numberwithin{pro}{section}
\numberwithin{exa}{section}
\numberwithin{rem}{section}
\numberwithin{lem}{section}
\numberwithin{defn}{section}
\numberwithin{cor}{section}

\numberwithin{equation}{section}

\begin{document}

\begin{flushright}
\today\\
arXiv:1306.2056
\end{flushright}

\begin{center}

\vspace*{2cm}

{\Large Extreme sizes in Gibbs-type exchangeable random partitions}

\vspace{2cm}

{\large
Shuhei Mano\footnote{Address for correspondence: The Institute of Statistical
Mathematics, 10-3 Midori-cho, Tachikawa, Tokyo 190-8562, Japan; 
Email: {\tt smano@ism.ac.jp}}

{\it The Institute of Statistical Mathematics\\
Tachikawa 190-8562, Japan}}

\vspace{2cm}

{\bf Abstract}

\end{center}

\vspace{.5cm}

Gibbs-type exchangeable random partitions, which is a class of multiplicative
measures on the set of positive integer partitions, appear in various contexts,
including Bayesian statistics, random combinatorial structures, and stochastic
models of diversity in various phenomena. Some distributional results on 
ordered sizes in the Gibbs partition are established by introducing associated
partial Bell polynomials and analysis of the generating functions. The 
combinatorial approach is applied to derive explicit results on asymptotic 
behavior of the extreme sizes in the Gibbs partition. Especially, Ewens-Pitman
partition, which is the sample from the Poisson-Dirichlet process and has been
discussed from rather model-specific viewpoints, and a random partition which 
was recently introduced by Gnedin, are discussed in the details. As 
by-products, some formulas for the associated partial Bell polynomials are 
presented.

\vspace{1cm}

\noindent
2010 Mathematics Subject Classification: 60C05, 05A17, 62G32
\vspace{.5cm}

\noindent
Keywords: random partition, extremes, analytic combinatorics, the Bell 
polynomials, Gibbs partitions, the Ewens-Pitman partition, the 
Poisson-Dirichlet process.

\newpage

\section{Introduction}

%
%

Exchangeable random partitions of a natural number appear in various contexts,
including Bayesian statistics, random combinatorial structures, and stochastic 
models of diversity in biological, physical, and sociological phenomena 
(see, for example, \cite{TavareEwens1997,Aoki2002,Arratia2003,Pitman2006}). 
The Gibbs-type exchangeable random partition is a class of multiplicative 
measures on the set of integer partitions \cite{GnedinPitman2005}. The Gibbs 
partition discussed in this paper (we will define precisely later) is probably
the broadest class of exchangeable random partitions ever appeared in 
literature; it covers many exchangeable random partitions proposed so far, 
including the Ewens-Pitman partition \cite{Ewens1972,Pitman1995}, which is the
sampling formula from the Poisson-Dirichlet process 
\cite{Ferguson1973,Kingman1975,Aldous1985,Pitman1995,PitmanYor1997},
the sampling formula from the normalized inverse-Gaussian process 
\cite{Lijoi2005}, the limiting conditional compound Poisson distribution 
\cite{Hoshino2009}, and a random partition proposed by Gnedin 
\cite{Gnedin2010}, which is a mixture of Dirichlet-multinomial distributions.

Let us define the Gibbs partition discussed in this paper. A partition of 
$[n]:=\{1,2,...,n\}$ into $k$ blocks is an unordered collection of non-empty
disjoint sets $\{A_1,...,A_k\}$ whose union is $[n]$. The multiset 
$\{|A_1|,...,|A_k|\}$ of unordered sizes of blocks of a partition $\pi_n$ of
$[n]$ defines a partition of integer $n$. The sequence of positive integer 
counts $(|\pi_n|_j, 1\le j\le n)$, where $|\pi_n|_j$ is the number of blocks in
$\pi_n$ of size $j$, with 
\begin{equation}
|\pi_n|:=\sum_{j=1}^n |\pi_n|_j=k, \qquad \sum_{j=1}^n j|\pi_n|_j=n.
\label{multiplicity}
\end{equation}
A random partition of $[n]$ is called exchangeable if its distribution is
invariant under the permutation of $[n]$. If for each partition 
$\{A_1,...,A_k\}$ of $[n]$
\begin{equation*}
{\mathbb P}(\Pi_n=\{A_1,...,A_k\})=p_n(|A_1|,...,|A_k|)
\end{equation*}
for some symmetric function of $p_n$ on a partition of $n$, the function $p_n$
is called the exchangeable partition probability function (EPPF). Suppose we 
have two sequences of non-negative weights, $(w_j)$ and $(v_{n,k})$ for 
$1\le j,k\le n$. A class of EPPF on the set of integer partition is called 
Gibbs form (Definition 1 of \cite{GnedinPitman2005}) if it admits a symmetric 
multiplicative representation 
\begin{equation}
p_n(n_1,...,n_k)=v_{n,k}\prod_{j=1}^k w_{n_j},
\label{Gibbs}
\end{equation}
or with multiplicities,
\begin{equation}
{\mathbb P}(|\Pi_n|_j=m_j,1\le j\le n)
=n!v_{n,k}\prod_{j=1}^n\left(\frac{w_j}{j!}\right)^{m_j}\frac{1}{m_j!}
\label{Gibbs_multi}
\end{equation}
for all $1\le k\le n$ and all partitions of $n$, where $\sum_j n_j=n$ or
$\sum_jm_j=k$ and $\sum_j jm_j=n$. 

There are two natural requirements for the Gibbs partitions. Gnedin and Pitman
\cite{GnedinPitman2005} showed that if the Gibbs form has {\it consistency} 
\cite{Kingman1978,Aldous1985,Pitman1995}, which means that a random partition 
of size $n$ is obtained from a random partition of size $n+1$ by discarding 
element $n+1$, $w$-weights have a form of
\begin{equation}
(w_\bullet)=((1-\alpha)_{\bullet-1}), \qquad -\infty<\alpha<1,
\label{Gibbs_cons}
\end{equation}
where, for real number $x$ and positive integer $i$, 
$(x)_i=x(x+1)\cdots(x+i-1)$ with a convention $(x)_0=1$. Since the consistency
is a reasonable property for applications, some of recent papers called the 
subclass of (\ref{Gibbs}) with the requirement of (\ref{Gibbs_cons}) the Gibbs
partitions (for example, \cite{GriffithsSpano2007,Lijoi2008,Gnedin2010}). In 
this paper we call the subclass the {\it consistent Gibbs partitions}. On the 
other hand, Pitman (Section~1.5 of \cite{Pitman2006}) called (\ref{Gibbs}) the 
Gibbs partition $Gibbs{}_{[n]}(v_\bullet,w_\bullet)$ if the $v$-weights are 
representable as ratios
\begin{equation}
v_{n,k}=\frac{v_k}{B_n(v_\bullet,w_\bullet)}, 
\qquad B_n(v_\bullet,w_\bullet):=\sum_{k=1}^nv_kB_{n,k}(w_\bullet),
\label{ratio}
\end{equation}
where
\begin{equation}
B_{n,k}(w_\bullet):=n!\sum_{\{m_\bullet:\sum m_j=k,\sum jm_j=n\}}\prod_{j=1}^n
\left(\frac{w_j}{j!}\right)^{m_j}\frac{1}{m_j!},\qquad n\ge k,
\label{Bell}
\end{equation}
and a convention $B_{n,k}(w_\bullet)=0$ for $n<k$, is the partial Bell 
polynomial in the variables $(w_\bullet)$. In addition, 
$Gibbs{}_{[n]}(v_\bullet,w_\bullet)$ is called Kolchin's model 
\cite{Kolchin1971,Kerov1995}, or $Gibbs{}_{[n]}(v_\bullet,w_\bullet)$ has 
Kolchin's representation \cite{Pitman2006}, which is identified with the 
collection of terms of random sum $X_1+\cdots+X_{|\Pi_n|}$ conditioned by 
$\sum_i X_i=n$ with independent and identically distributed $X_1, X_2,...,$ 
independent of $|\Pi_n|$. As we have seen, definition of Gibbs partition 
depends on contexts, authors, and papers. To make our discussion has most 
generality, throughout this paper we call the class of exchangeable random 
partitions whose EPPF have the form of (\ref{Gibbs}) the Gibbs partition. In 
this paper we will discuss properties of the broadest class.

If a Gibbs partition $Gibbs{}_{[n]}(v_\bullet,w_\bullet)$ has the consistency,
it reduces to the Ewens-Pitman partition \cite{Kerov1995,GnedinPitman2005}. 
The Ewens-Pitman partition appears in various contexts, whose classical 
examples include cycle lengths in random permutation~\cite{SheppLloyd1966}, a 
sample from the infinite many allele model in population 
genetics~\cite{Ewens1972}, and a sample from the Dirichlet process prior in 
Bayesian nonparametrics~\cite{Antoniak1974}. The $v$-weights have a form of
\begin{equation}
v_k=(\theta)_{k;\alpha},\qquad B_n(v_\bullet,w_\bullet)=(\theta)_n,
\label{v_EP}
\end{equation}
where for real numbers $x$ and $a$ and positive integer $i$,
$(x)_{i;a}=x(x+a)\cdots(x+(i-1)a)$ with a convention $(x)_{0;a}=1$. The pair of
real parameters $\alpha$ and $\theta$ satisfy either $0\le\alpha<1$ and 
$\theta >-\alpha$, or $\alpha<0$ and $\theta=-m\alpha$, $m=1,2,...$. For 
$\alpha<0$ the Ewens-Pitman partition reduces to a symmetric 
Dirichlet-multinomial distribution of parameter $(-\alpha)$. Another example of
consistent Gibbs partitions discussed in this paper is Gnedin's partition 
\cite{Gnedin2010}, which is a mixture of Dirichlet-multinomial distributions.
In contrast to the Ewens-Pitman partition $v$-weights of Gnedin's partition are
not representable as ratios; in Gnedin's partition $w$-weights satisfy 
(\ref{Gibbs_cons}) with $\alpha=-1$ and the $v$-weights have a form of
\begin{equation}
v_{n,k}=(\gamma)_{n-k}\frac{\prod_{j=1}^{k-1}(j^2-\gamma j+\zeta)}
{\prod_{j=1}^{n-1}(j^2+\gamma j+\zeta)},
\label{v_gnedin}
\end{equation}
where $\zeta$ and $\gamma$ are chosen such that $\gamma\ge 0$ and
$j^2-\gamma j+\zeta>0$ for $j\ge 1$.

%
%

In studies of exchangeable random partitions the Ewens-Pitman partition has
played the central role and discussion on the generalizations has been only 
recently started. The Ewens-Pitman partition and the Poisson-Dirichlet process,
which is closely related to the Ewens-Pitman partition, have nice properties.
For example, if $\alpha=0$ and $\theta>0$ the Ewens-Pitman partition satisfies
the {\it conditioning relation}:
\begin{equation}
(|\Pi_n|_1,...,|\Pi_n|_n)\overset{d}{=}(Z_1,...,Z_n)|\sum_{j=1}^njZ_j=n,
\label{cond_rel}
\end{equation}
where $Z_j$, $j=1,...,n$, independently follow the Poisson distribution of 
parameter $\theta/j$. Moreover, the multiplicities of the small components
are asymptotically independent:
\begin{equation}
(|\Pi_n|_1,|\Pi_n|_2,...)\overset{d}{\to}(Z_1,Z_2,...),\qquad n\to\infty.
\label{joint_Ewens}
\end{equation}
Approaches based on the conditioning relation are very powerful if a random 
combinatorial structure has the property (see \cite{Arratia2003} for a
comprehensive survey). Unfortunately, the asymptotic independence 
(\ref{joint_Ewens}) does not hold even in the Ewens-Pitman partition with 
non-zero $\alpha$. Studies of the Ewens-Pitman partition has been heavily 
depend on properties of the Poisson-Dirichlet process (see, for example, 
\cite{Arratia2003,Pitman2006}). The dependency on the Poisson-Dirichlet process
has made arguments model-specific. The connection with Poisson-Dirichlet 
process might have little use in studies of general Gibbs partitions. On the 
other hand, in studies of random combinatorial structures analytic 
combinatorial approaches have been quite useful in applications to various
problems \cite{FlajoletOdlyzko1990,FlajoletSedgewick2009}. For example, 
Panario and Richmond \cite{PanarioRichmond2001} showed that in a decomposition
of a random permutation into cycles, which corresponds to the Ewens-Pitman 
partition of parameter $(\alpha,\theta)=(0,1)$, the singularity analysis of 
the generating function, which is a popular tool in analytic combinatorics, 
yields asymptotic behavior of the ordered cycle lengths. In the present paper 
we will see that analytic combinatorial approaches are general enough to apply
to a broader class of exchangeable random partitions, namely, the Gibbs 
partition.

Behavior of the extreme sizes in random partitions is a classic issue. Some 
examples from statistical application are an exact test for the maximum 
component in a periodgram by Fisher \cite{Fisher1929}, and an exact test of 
natural selection operating on the most frequent gene type in population 
genetics \cite{Ewens1973}. In addition, as a measure of diversity, distribution
of the maximum size has been discussed in, for example, population genetics 
\cite{WattersonGuess1977} and economics \cite{Aoki2002}. Asymptotic behavior of
extreme sizes has been attracted many authors, not only by practical importance
but also by mathematical interest. Asymptotic behavior of the ordered cycle 
length in a random permutation has been discussed in \cite{SheppLloyd1966}. In
the number theory, a number whose largest prime factor is not larger than $x$ 
is called $x$-smooth number, while a number whose smallest prime factor is 
larger than $y$ is called $y$-rough number (see, for example, Chapters III.5 
and III.6 of \cite{Tenenbaum1995}). The limiting distributions of the counting
functions of the smooth number and the rough number are coincidentally 
identical to the distribution functions of the extreme sizes in the 
Ewens-Pitman partition of parameters $(\alpha,\theta)=(0,1)$
\cite{Arratia2003}. An extension to the case of $\theta\neq 1$ in the context 
of the number theory was also discussed in \cite{Hensley1984}. For the smallest
sizes in the Ewens-Pitman partition of parameters $(\alpha,\theta)=(0,\theta)$,
it is known that the probability that the smallest size is larger than 
$r\asymp n\to\infty$ involves a generalization of Buchstab's function in the 
number theory (Corollary \ref{cor:LK_dist_E_On}) and the probability that the 
smallest size is larger than $r=o(n)$ follows immediately from the asymptotic 
independence (\ref{joint_Ewens}), which is restricted to the case that 
$\alpha=0$. In this paper we will see how these properties are in the 
Ewens-Pitman partition of parameters $0<\alpha<1$ and $\theta>-\alpha$. We will
see that we do not have ``generalized Buchstab's function'' 
(Corollary~\ref{cor:LK_dist_P_On}) and we will establish 
(Theorem~\ref{thm:LK_dist_P_on}) a precise asymptotic of the probability that 
the smallest size is larger than $r=o(n)$: 
$f_{\alpha,\theta}(r) n^{-\alpha-\theta}$ with some explicit function 
$f_{\alpha,\theta}(r)$. The asymptotic independence no longer holds, 
nevertheless, the singularity analysis of the generating function in analytic
combinatorics gives the estimate straightforwardly.

%
%

This paper is organized as follows. In Section~2 we introduce associated 
partial Bell polynomials by restricting sizes of components in enumerating 
possible partitions. In Section~3 we obtain some distributional results of 
ordered sizes in the Gibbs partition in terms of the generating functions.
In Section~4 asymptotic behavior of extreme sizes is discussed. We see explicit
results on asymptotic behavior of extreme sizes in the consistent Gibbs 
partition, which includes the Ewens-Pitman partition and Gnedin's partition.
Some of the results regarding the Ewens-Pitman partition are reproductions of 
known results, however, we demonstrate that our alternative derivation is 
simpler than approaches using model-specific properties of the Ewens-Pitman
partition. In principle, the developed approach is able to apply to any type of
Gibbs partition. Section~5 is devoted to the proofs.

A computer program to generate Gibbs partitions is available upon request to
the author.

\section{Associated partial Bell polynomials}\label{sect:bell}

Let us begin with a proposition on the partial Bell polynomial (\ref{Bell}), 
which follows immediately from Fa\`{a} di Bruno's formula 
\cite{Comtet1974, Pitman2006}:
\begin{equation}
B_n(v_\bullet,w_\bullet)
=\left[\frac{\xi^n}{n!}\right]\breve{v}(\breve{w}(\xi)),
\end{equation}
where $\breve{v}$ and $\breve{w}$ are the exponential generating functions:
\begin{equation*}
\breve{v}(\eta):=\sum_{j=1}^\infty v_j\frac{\eta^j}{j!},\qquad
\breve{w}(\xi):=\sum_{j=1}^\infty w_j\frac{\xi^j}{j!},
\end{equation*}
and $[\xi^n/n!]f(\xi):=a_n$ for a series $f(\xi)=\sum_j a_j\xi^j/j!$. 

\begin{pro}\label{pro:G_def}\rm 
The exponential generating functions of the partial Bell polynomials are
\begin{equation}
\breve{B}_k(\xi,w_\bullet):=\sum_{n=k}^\infty B_{n,k}(w_\bullet)\frac{\xi^n}
{n!}=\frac{(\breve{w}(\xi))^k}{k!}, \qquad k=1,2,...
\label{G_egf}
\end{equation}
\end{pro}

\begin{exa}\rm
Setting $(w_\bullet)=((\bullet-1)!)$ yields
\begin{equation*}
\breve{B}_k(\xi,(\bullet-1)!)=\frac{\{-\log(1-\xi)\}^k}{k!},\qquad k=1,2,...
\end{equation*}
where $B_{n,k}((\bullet-1)!)\equiv|s(n,k)|$ are the signless Stirling number of
the first kind defined by \cite{Charalambides2005}
\begin{equation}
B_n(\theta^\bullet,(\bullet-1)!)=
(\theta)_n=\sum_{k=0}^n|s(n,k)|\theta^k,\qquad n=0,1,...
\end{equation}
\end{exa}

\begin{exa}\label{exa:C_def}\rm
For non-zero $\alpha$, setting $(w_\bullet)=([\alpha]_\bullet)$ yields
\begin{equation}
\breve{B}_k(\xi,[\alpha]_{\bullet})
=\frac{\left\{(1+\xi)^\alpha-1\right\}^k}{k!}, \qquad k=1,2,...,
\label{C_egf}
\end{equation}
where for real number $x$ and positive integer $i$, 
$[x]_i=x(x-1)\cdots(x-i+1)$ with a convention $[x]_0=1$, and 
$B_{n,k}([\alpha]_\bullet)\equiv C(n,k;\alpha)$ are the generalized factorial 
coefficients introduced by \cite{Charalambides2005}
\begin{equation}
B_n([x]_\bullet,[\alpha]_\bullet)=[\alpha x]_n=\sum_{k=0}^nC(n,k;\alpha)[x]_k,
\qquad n=0,1,...
\label{C_sum}
\end{equation}
\end{exa}

\noindent
In general, for distinct real numbers $a$ and $b$, 
$B_{n,k}([b-a]_{(\bullet-1);a})\equiv S^{a,b}_{n,k}$ are generalized Stirling 
numbers defined by
\cite{Pitman2006}
\begin{equation*}
B_n([x]_{\bullet;b},[b-a]_{(\bullet-1);a})=[x]_{n;a}
=\sum_{k=0}^nS^{a,b}_{n,k}[x]_{k;b}, \qquad n=0,1,...,
\end{equation*}
where for real number $x$ and $a$ and positive integer $i$, 
$[x]_{i;a}=x(x-a)\cdots(x-(i-1)a)$  with a convention $[x]_{0;a}=1$. For 
example, $|s(n,k)|=(-1)^{n+k}S^{1,0}_{n,k}$ and 
$C(n,k;\alpha)=\alpha^kS^{1,\alpha}_{n,k}$.

Dropping off first terms in the sequence $(w_\bullet)$ of 
Proposition~\ref{pro:G_def} gives a modified version of
Proposition~\ref{pro:G_def}, which introduces the {\it associated partial Bell
polynomials}. We call the polynomials {\it associated} since they cover 
associated numbers appear in combinatorics literature, such as the associated
signless Stirling numbers of the first kind and the associated generalized 
factorial coefficients (see, for example, \cite{Comtet1974,Charalambides2005}).
The associated partial Bell polynomials play the central role throughout this
paper.

\begin{pro}\label{pro:Gr_def}\rm
For a sequence $(w_j)$, $1\le j<\infty$, and positive integer $r$, set 
$w_{(r)j}=0$, $1\le j\le r-1$ and $w_{(r)j}=w_j$, $j\ge r$. Define the 
associated partial Bell polynomials
\begin{equation*}
B_{n,k,(r)}(w_{\bullet}):=B_{n,k}(w_{(r)\bullet})
=n!\sum_{\substack{\{m_\bullet:\sum m_j=k,\sum jm_j=n,\\m_{j<r}=0\}}}
\prod_{j=1}^n\left(\frac{w_j}{j!}\right)^{m_j}\frac{1}{m_j!},\qquad n\ge rk,
\end{equation*}
with a convention $B_{n,k,(r)}(w_{\bullet})=0$ for $n<rk$ and 
$B_{n,k,(1)}(w_{\bullet})=B_{n,k}(w_\bullet)$. Then, the exponential generating
function of the sequence $(w_{(r)\bullet})$, $\breve{w}_{(r)}(\xi)$, provides 
the exponential generating functions of the associated partial Bell polynomials
\begin{equation}
\breve{B}_{k,(r)}(\xi,w_\bullet)
:=\sum_{n=rk}^\infty B_{n,k,(r)}(w_{\bullet})\frac{\xi^n}{n!}=
\frac{(\breve{w}_{(r)}(\xi))^k}{k!}, \qquad k=1,2,...
\label{Gr_egf}
\end{equation}
\end{pro}

\begin{exa}\label{exa:sr_def}\rm
Setting $(w_\bullet)=((\bullet-1)!)$ yields
\begin{equation*}
\breve{B}_{k,(r)}(\xi,(\bullet-1)!)=
\frac{1}{k!}\left\{-\log(1-\xi)-\sum_{j=1}^{r-1}\frac{\xi^j}{j}\right\}^k,
\qquad k=1,2,...,
\end{equation*}
where $B_{n,k,(r)}((\bullet-1)!)\equiv|s_r(n,k)|$ are known as the 
$r$-associated signless Stirling numbers of the first kind 
\cite{Comtet1974,Charalambides2005}. The associated signless Stirling number of
the first kind has an interpretation in terms of a decomposition of a random 
permutation into cycles. In decomposing a set of $n$ elements into $k$ cycles 
the number of permutations in which each length of cycle is not shorter than 
$r$ is $|s_r(n,k)|$.
\end{exa}

\begin{exa}\label{exa:Cr_def}\rm
For non-zero $\alpha$, setting $(w_\bullet)=([\alpha]_\bullet)$ yields
\begin{equation}
\breve{B}_{k,(r)}(\xi,[\alpha]_{\bullet})=
\frac{1}{k!}\left\{(1+\xi)^\alpha-\sum_{j=0}^{r-1}
\left(\begin{array}{c}\alpha\\j\end{array}\right)\xi^j\right\}^k,
\qquad 
k=1,2,...,
\label{Cr_egf}
\end{equation}
where $B_{n,k,(r)}([\alpha]_{\bullet})\equiv C_r(n,k;\alpha)$ are known as the 
$r$-associated generalized factorial coefficients \cite{Charalambides2005}.
Suppose that $n$ like balls are distributed into $k$ distinguishable urns, each
with $\alpha\,(\ge n)$ distinguishable cells whose capacity is limited to one 
ball. The enumerator for occupancy is
\begin{equation*}
\sum_{j=1}^\alpha\left(\begin{array}{c}\alpha\\j\end{array}\right)
\xi^j=(1+\xi)^\alpha-1,
\end{equation*}
and the generating function for occupancy of the $k$ urns satisfies
\begin{equation*}
\sum_{n=k}^{\alpha k}A(n,k;\alpha)\xi^n=\left\{(1+\xi)^\alpha-1\right\}^k.
\end{equation*}
Comparing with the exponential generating function of the generalized factorial
coefficients (\ref{C_egf}) implies that the number of different distributions 
of $n$ like balls into $k$ distinguishable urns, each with $\alpha$ 
distinguishable cells of occupancy limited to one ball, equals 
$A(n,k;\alpha)=k!C(n,k;\alpha)/n!$. If each urn is occupied by at least $r$ 
balls, the enumerator for occupancy of an urn is
\begin{equation*}
\sum_{j=r}^\alpha
\left(\begin{array}{c}\alpha\\j\end{array}\right)\xi^j
\end{equation*}
and the generating function for occupancy of the $k$ urns satisfies
\begin{equation*}
\sum_{n=rk}^{\alpha k}A_r(n,k;\alpha)\xi^n
=\left\{(1+\xi)^\alpha-\sum_{j=0}^{r-1}
\left(\begin{array}{c}\alpha\\j\end{array}\right)\xi^j\right\}^k.
\end{equation*}
Comparing with the exponential generating function of the associated 
generalized factorial coefficients (\ref{Cr_egf}) implies that the number
of different distributions of $n$ like balls into $k$ distinguishable urns, 
each with $\alpha$ distinguishable cells of occupancy limited to one ball, so 
that each urn is occupied by at least $r$ balls equals 
$A_r(n,k;\alpha)=k!C_r(n,k;\alpha)/n!$.
\end{exa}

When the sequence $(w_\bullet)$ is truncated we have another modified version 
of Proposition~\ref{pro:G_def}. The following proposition provides another kind
of associated partial Bell polynomials. The author is unaware of literature 
where this type of associated combinatorial numbers are discussed, but they 
will play important roles in this paper. The enumerating interpretations are 
similar to those in Examples~\ref{exa:sr_def} and \ref{exa:Cr_def}.

\begin{pro}\label{pro:G^r_def}\rm
For a sequence $(w_\bullet)$, and positive integer $r$, set $w^{(r)}_j=w_j$, 
$1\le j\le r$ and $w^{(r)}_j=0$, $j\ge r+1$. Define the associated partial Bell
polynomials
\begin{equation*}
B_{n,k}^{(r)}(w_\bullet):=B_{n,k}(w^{(r)}_\bullet)
=n!\sum_{\substack{\{m_\bullet:\sum m_j=k,\sum jm_j=n,\\m_{j>r}=0\}}}
\prod_{j=1}^n\left(\frac{w_j}{j!}\right)^{m_j}\frac{1}{m_j!},
\qquad k\le n\le rk,
\end{equation*}
with a convention $B^{(r)}_{n,k}(w_\bullet)=0$ for $n<k$ and $n>rk$. We have 
$B^{(r)}_{n,k}(w_\bullet)=B_{n,k}(w_\bullet)$, $r\ge n-k+1$. Then, the 
exponential generating function of the sequence $(w^{(r)}_\bullet)$, 
$\breve{w}^{(r)}(\xi)$, provides the exponential generating functions of the 
associated partial Bell polynomials
\begin{equation}
\breve{B}_k^{(r)}(\xi,w_\bullet)
:=\sum_{n=k}^{rk}B_{n,k}^{(r)}(w_\bullet)\frac{\xi^n}{n!}
=\frac{(\breve{w}^{(r)}(\xi))^k}{k!},\qquad k=1,2,...,
\label{G^r_egf}
\end{equation}
\end{pro}

In applications, especially for large $n$, recurrence relations are inevitable
to compute the associated partial Bell polynomials introduced above. We provide
some recurrence relations for the associated partial Bell polynomials in 
Appendix~A.

Further modification of Proposition~\ref{pro:G_def} provides another kind of 
associated partial Bell polynomials. As a natural extension of 
Propositions~\ref{pro:Gr_def} and \ref{pro:G^r_def} is consideration of a set 
of $n$ elements into $k$ blocks so that the size of the $i$-th largest block is
not larger than $r$. Following proposition gives the extension. The proof is 
provided in Subsection \ref{subsec:proof_Bell}.

\begin{pro}\label{pro:B^{ri}nk}
Let the exponential generating functions $\breve{w}_{(r)}$ and 
$\breve{w}^{(r)}$ be defined as Propositions~\ref{pro:Gr_def} and
\ref{pro:G^r_def}. Let us define associated Bell polynomials by 
\begin{equation*}
B^{(r),(i)}_{n,k}(w_\bullet):=n!
\sum_{\substack{\{m_\bullet:\sum m_j=k,\sum jm_j=n,\\m_{r+1}+\cdots+m_n<i\}}} 
\prod_{j=1}^r\left(\frac{w_j}{j!}\right)^{m_j}\frac{1}{m_j!},\qquad n\ge k,
\end{equation*}
with a convention $B^{(r),(i)}_{n,k}(w_\bullet)=0$, $n<k$, for $2\le i\le k$, 
and $B^{(r),(1)}_{n,k}(w_\bullet)=B^{(r)}_{n,k}(w_\bullet)$. Then, the 
exponential generating function is given by
\begin{equation*}
\breve{B}_k^{(r),(i)}(\xi,w_\bullet)
:=\sum_{n=k}^\infty B^{(r),(i)}_{n,k}(w_\bullet)\frac{\xi^n}{n!}
=\sum_{j=0}^{i-1}\breve{B}_{j,(r+1)}(\xi,w_\bullet)\breve{B}_{k-j}^{(r)}
(\xi,w_\bullet).
\end{equation*}
\end{pro}

\noindent
Proposition~\ref{pro:B^{ri}nk} means that an associated partial Bell 
polynomials $B^{(r),(i)}_{n,k}(w_\bullet)$ is representable as quadratic 
polynomial in the associated partial Bell polynomials $B_{n,k,(r)}(w_\bullet)$ 
and $B^{(r)}_{n,k}(w_\bullet)$. Moreover, the next proposition, whose proof is
in Subsection~\ref{subsec:proof_Bell}, implies that $B_{n,k,(r)}(w_\bullet)$ 
and $B^{(r)}_{n,k}(w_\bullet)$ can be expressed in terms of the partial Bell 
polynomials $B_{n,k}(w_\bullet)$. Therefore in principle all associated partial
Bell polynomials introduced in this paper can be expressed in terms of the 
partial Bell polynomials.

\begin{pro}\label{pro:GA_byG}
For positive integer $r$ and $k$, the associated partial Bell polynomials,
$B^{(r)}_{n,k}(w_\bullet)$, satisfy, if $r+k\le n\le rk$,
\begin{eqnarray}
B^{(r)}_{n,k}(w_\bullet)&=&B_{n,k}(w_\bullet)\nonumber\\
&&+\sum_{l=1}^{\lfloor (n-k)/r\rfloor}\frac{(-1)^l}{l!}
\sum_{\substack{i_1,...,i_l\ge r+1,\\i_1+\cdots+i_l\le n-k+l}}
B_{n-(i_1+\cdots+i_l),k-l}(w_\bullet)[n]_{i_1+\cdots+i_l}
\prod_{j=1}^l\frac{w_{i_j}}{i_j!}
\label{G^r_byG}
\end{eqnarray}
and $B^{(r)}_{n,k}(w_\bullet)=B_{n,k}(w_\bullet)$ if $k\le n\le r+k-1$. For 
positive integer $r$ and $k$, the associated partial Bell polynomials, 
$B_{n,k,(r)}(w_\bullet)$, satisfy, if $n\ge rk$,
\begin{eqnarray}
B_{n,k,(r)}&=&B_{n,k}(w_\bullet)\nonumber\\
&&+\sum_{l=1}^{k-1}\frac{(-1)^l}{l!}
\sum_{\substack{1\le i_1,...,i_l\le r-1,\\i_1+\cdots+i_l\le n-k+l}}
B_{n-(i_1+\cdots+i_l),k-l}(w_\bullet)[n]_{i_1+\cdots+i_l}  
\prod_{j=1}^l\frac{w_{i_j}}{i_j!}.
\label{Gr_byG}
\end{eqnarray}
\end{pro}

\section{Ordered sizes in the Gibbs partition}\label{sec:ordered_sizes}

It is straightforward to obtain some explicit distributional results on the 
ordered sizes in the Gibbs partition of the form (\ref{Gibbs}) in terms of the
associated partial Bell polynomials. Denote the descending order statistics of
the block sizes by $|A_{(1)}|,...,|A_{(|\Pi_n|)}|$, where 
$|A_{(1)}|\ge|A_{(2)}|\ge\cdots\ge|A_{(|\Pi_n|)}|$. The distribution of the 
number of blocks in (\ref{Gibbs_multi}) follows immediately 
\cite{GnedinPitman2005}
\begin{equation*}
{\mathbb P}(|\Pi_n|=k)=
\sum_{\{m_\bullet:\sum m_j=k,\sum jm_j=n\}}
{\mathbb P}(|\Pi_n|_j=m_j,1\le j\le n)=v_{n,k}B_{n,k}(w_\bullet), \qquad 1\le 
k \le n.
\end{equation*}
The conditional distribution given the number of blocks is
\begin{equation}
{\mathbb P}(|\Pi_n|_j=m_j,1\le j\le n||\Pi_n|=k)=\frac{n!}{B_{n,k}(w_\bullet)}
\prod_{j=1}^n\left(\frac{w_j}{j!}\right)^{m_j}\frac{1}{m_j!}, 
\qquad 1\le k \le n.
\label{micro_Gibbs}
\end{equation}
In statistical mechanics this is a microcanonical Gibbs distribution function
whose number of microstates of a block of size $j$ is $w_j$
\cite{BerestyckiPitman2007}. For Gibbs partitions the number of blocks is the
sufficient statistics for $v$-weights. By virtue of the sufficiency, discussion
on the ordered sizes reduces to enumeration of microstates of the 
microcanonical Gibbs distribution which fulfills a given condition. The 
definitions of the associated partial Bell polynomials introduced in the 
previous section were defined by such enumeration. For example, the 
distribution of the largest size conditioned by the number of blocks is
\begin{equation}
{\mathbb P}(|A_{(1)}|\le r||\Pi_n|=k)=\frac{B_{n,k}^{(r)}(w_\bullet)}
{B_{n,k}(w_\bullet)}, \qquad 1\le k\le n,\qquad n/k\le r\le n,
\label{L1_dist_microG}
\end{equation}
and ${\mathbb P}(|A_{(1)}|\le r||\Pi_n|=k)=0$ for $1\le r<n/k$. Note that the 
associated partial Bell polynomial, $B_{n,k}^{(r)}(w_\bullet)$, is the number 
of microstates in the microcanonical Gibbs distribution of the form 
(\ref{micro_Gibbs}) whose largest size is not larger than $r$. The marginal 
distributions of the ordered sizes have following representation.

\begin{lem}\label{lem:L_dist_G}
In a Gibbs partition of the form $(\ref{Gibbs})$ the marginal distributions of 
the ordered sizes are
\begin{eqnarray}
{\mathbb P}(|A_{(1)}|\le r)
&=&\sum_{k=\lceil n/r\rceil}^nv_{n,k}B^{(r)}_{n,k}(w_\bullet),
\label{L1_dist_G}\\
{\mathbb P}(|A_{(i)}|\le r)&=&\sum_{k=1}^{i-1}v_{n,k}B_{n,k}(w_\bullet)+
\sum_{k=i}^nv_{n,k}B^{(r),(i)}_{n,k}(w_\bullet), \qquad 2\le i\le n,\nonumber
\end{eqnarray}
and
\begin{equation}
{\mathbb P}(|A_{(|\Pi_n|)}|\ge r)
=\sum_{k=1}^{\lfloor n/r\rfloor}v_{n,k}B_{n,k,(r)}(w_\bullet)
\label{LK_dist_G},
\end{equation}
for $1\le r\le n$, where $|A_{(i)}|=0$ if $i>|\Pi_n|$.
\end{lem}

\noindent
Hence, discussion on the ordered sizes reduces to analysis of the partial Bell
polynomials and the mixtures of them by $v$-weights. Distributions of the 
extremes are representable as composition of the exponential generating 
functions; substituting the exponential generating functions (\ref{G^r_egf}) 
and (\ref{Gr_egf}) into (\ref{L1_dist_G}) and (\ref{LK_dist_G}), respectively,
yields
\begin{equation}
{\mathbb P}(|A_{(1)}|\le r)=\left[\frac{\xi^n}{n!}\right]
\breve{v}_n(\breve{w}^{(r)}(\xi)),
\label{L1_dist_Gg}
\end{equation}
and
\begin{equation}
{\mathbb P}(|A_{(|\Pi_n|)}|\ge r)=\left[\frac{\xi^n}{n!}\right]
\breve{v}_n(\breve{w}_{(r)}(\xi)),
\label{LK_dist_Gg}
\end{equation}
where $\breve{v}_n$ is the exponential generating function of 
$(v_{n,j})$-weights:
\begin{equation*}
\breve{v}_n(\xi):=\sum_{j=1}^\infty v_{n,j}\frac{\xi^j}{j!}.
\end{equation*}

\begin{rem}\rm
A Gibbs partition $Gibbs{}_{[n]}(v_\bullet,w_\bullet)$, in which $v$-weights 
are representable as ratios (\ref{ratio}), has Kolchin's representation 
\cite{Pitman2006}, which is identified with the collection of terms of random 
sum $X_1+\cdots+X_{|\Pi_n|}$ conditioned on $\sum_i X_i=n$ with independent and
identically distributed $X_1, X_2,...$ independent of $|\Pi_n|$ 
\cite{Kerov1995}. Here, the (ordinary) probability generating function of 
$X_\bullet$ is $\breve{w}(\xi)/\breve{w}(1)$ and the probability generating 
function of the random sum $X_1+\cdots+X_{|\Pi_n|}$ is 
$\breve{v}(\breve{w}(\xi))/\breve{v}(\breve{w}(1))$. In 
Propositions~\ref{pro:Gr_def} and \ref{pro:G^r_def} $\breve{w}_{(r)}$ and 
$\breve{w}^{(r)}$ are introduced by dropping terms from the sequence of 
$w$-weights. Therefore $\breve{w}_{(r)}$ and $\breve{w}^{(r)}$ are the 
probability generating functions of defective distributions induced by a proper
probability mass function $(w_\bullet/\bullet!)$. Suppose independent and 
identically distributed random variables $X_{(r)1}$, $X_{(r)2}$,..., whose 
probability generating function are $w_{(r)}$, and $X^{(r)}_1$, $X^{(r)}_2$,...
whose probability generating function are $w^{(r)}$. Then,
\begin{equation*}
{\mathbb P}(|A_{(1)}|\le r)
=\frac{{\mathbb P}(X_1^{(r)}+\cdots+X_{|\Pi_n|}^{(r)} =n)}
{{\mathbb P}(X_1+\cdots+X_{|\Pi_n|}=n)},
\end{equation*}
and
\begin{equation*}
{\mathbb P}(|A_{(|\Pi_n|)}|\ge r)
=\frac{{\mathbb P}(X_{(r)1}+\cdots+X_{(r)|\Pi_n|}=n)}
{{\mathbb P}(X_1+\cdots+X_{|\Pi_n|}=n)}.
\end{equation*}
Hence, distributions of the extreme sizes are the ratios of the probability 
mass of random sum at $n$ in the defective distribution to that in the proper 
distribution.
\end{rem}

\section{Asymptotic behavior of extreme sizes}

Asymptotic behavior of the extreme sizes in the Ewens-Pitman partition has been
discussed in various contexts (see, for example,
\cite{SheppLloyd1966,Watterson1976,Griffiths1988,ArratiaTavare1992,
PitmanYor1997,PanarioRichmond2001,Arratia2003,Handa2009}). In this section we 
discuss asymptotic behavior of the extreme sizes in general Gibbs partitions of
the form (\ref{Gibbs}). The developed Lemma~\ref{lem:L_dist_G} is useful for 
keeping generality of our discussion, since it holds in any Gibbs partition of
the form (\ref{Gibbs}). Then, some explicit results for the consistent Gibbs 
partition, which is a class of Gibbs partitions whose $w$-weights have a form 
of (\ref{Gibbs_cons}), are presented. Subsequently, further explicit results 
are presented for two specific examples of consistent Gibbs partitions, the 
Ewens-Pitman partition and Gnedin's partition. 

Explicit asymptotic forms appear in this section involve an extension of 
incomplete Dirichlet integrals, which involves a {\it Dirichlet distribution 
with negative parameters}. Let us prepare some notations. The probability 
density of a Dirichlet distribution of $b+1$ variables parametrized by two 
parameters $\rho>0$ and $\nu>0$ is
\begin{equation*}
p(y_1,y_2,...,y_{b+1})
=\frac{\Gamma(\rho+b\nu)}{\Gamma(\rho)\Gamma(\nu)^b}y_{b+1}^{\rho-1}
\prod_{j=1}^by_j^{\nu-1}, \qquad \sum_{j=1}^{b+1}y_j=1,
\end{equation*}
whose support is the $b$-dimensional simplex 
$\Delta_b:=\{y_i:0<y_i,1\le i \le b, \sum_{j=1}^b y_j<1\}$. Incomplete 
Dirichlet integrals are usually defined in this setting \cite{Sobel1977}. But 
let us introduce an integral with non-zero real parameters $\rho$ and $\nu$:
\begin{equation*}
{\mathcal I}_{p,q}^{(b)}(\nu;\rho)
:=\frac{\Gamma(\rho+b\nu)}{\Gamma(\rho)\Gamma(\nu)^b}
\int_{\Delta_b(p,q)}y_{b+1}^{\rho-1}\prod_{j=1}^b y_j^{\nu-1}dy_j,
\end{equation*}
with a convention
\begin{equation*}
{\mathcal I}_{p,q}^{(b)}(0;\rho)
:=\int_{\Delta_b(p,q)}y_{b+1}^{\rho-1}\prod_{j=1}^b y_j^{-1}dy_j
\end{equation*}
and ${\mathcal I}^{(0)}_{p,q}(\nu;\rho)=1$, where
\begin{equation*}
\Delta_b(p,q):=\left\{y_i:p<y_i,1\le i\le b; \sum_{j=1}^b y_j<1-q\right\},
\qquad 0<q<1, \qquad 0<p<\frac{1-q}{b}.
\end{equation*} 
Of course, when either of $\rho$ and $\nu$ is negative the integral over the 
simplex $\Delta_b$ does not exist. But throughout this paper integrals 
involving this extension of incomplete Dirichlet integrals are well defined 
since the domain of integration, $\Delta_b(p,q)$, is appropriately chosen.

\subsection{General distributional results}\label{subsec:dist_res}

Let us begin with seeing asymptotic behavior of the smallest sizes in the Gibbs
partition of the form (\ref{Gibbs}). Lemma~\ref{lem:L_dist_G} and the 
Cauchy-Goursat theorem provide a way to evaluate it in terms of a contour 
integral. This kind of method to evaluate asymptotics, which is called the 
singularity analysis of generating functions in analytic combinatorics 
literature, has been quite popular in studies of random combinatorial 
structures (see, for example, \cite{FlajoletSedgewick2009}). Noting the 
expression (\ref{LK_dist_Gg}) the distribution of the smallest size can be 
evaluated as
\begin{equation}
{\mathbb P}(|A_{(|\Pi_n|)}|\ge r)=\frac{n!}{2\pi\sqrt{-1}}\oint
\frac{\breve{v}_n(\breve{w}_{(r)}(\xi))}{\xi^{n+1}}d\xi, \qquad n\to\infty, 
\qquad r=o(n).
\label{LK_dist_G_on}
\end{equation}

It is interesting to see the event that the smallest size is extremely large. 
To see the asymptotic behavior we need asymptotic forms of the $v$-weights and
the associated partial Bell polynomials, $B_{n,k,(r)}(w_\bullet)$, in 
$n,r\to\infty$ with $r\asymp n$ and fixed $k$. Explicit results are immediately
deduced by substituting these asymptotic forms into (\ref{LK_dist_G}). 

Asymptotic behavior of the largest size in the Gibbs partition of the form 
(\ref{Gibbs}) can be discussed similarly. The expression (\ref{L1_dist_Gg})
leads
\begin{equation}
{\mathbb P}(|A_{(1)}|\le r)=\frac{n!}{2\pi\sqrt{-1}}\oint
\frac{\breve{v}_n(\breve{w}^{(r)}(\xi))}{\xi^{n+1}}d\xi, \qquad n\to\infty, 
\qquad r=o(n).
\label{L1_dist_G_on}
\end{equation}

It is also interesting to see the event that the largest size is extremely 
small. Following lemma, whose proof is in Subsection~\ref{subsec:proof_G}, 
provides asymptotic expressions of the marginal distribution of the largest 
size in terms of the partial Bell polynomials.

\begin{lem}\label{lem:L1_dist_G_On}
In a Gibbs partition of the form $(\ref{Gibbs})$ whose weights and induced
partial Bell polynomials satisfy
\begin{equation*}
\frac{w_n}{n!}=O(n^{-1-\eta_1}),\qquad n!v_{n,k}=O(n^{1-\eta_2(k)}),\qquad
\frac{B_{n,k}(w)}{n!}=O(n^{-1-\eta_3(k)}),\qquad n\to\infty,
\end{equation*}
for fixed positive integer $k$, the largest size satisfies
\begin{eqnarray}
{\mathbb P}(|A_{(1)}|\le r)&=&
1+\sum_{l=1}^{\lfloor(n-\lceil n/r\rceil)/r\rfloor}\frac{(-1)^l}{l!}
\sum_{\substack{i_1,...,i_l\ge r+1,\\i_1+\cdots+i_l\le n-\lceil n/r\rceil+l}}
[n]_{i_1+\cdots+i_l}\prod_{j=1}^l\frac{w_{i_j}}{i_j!}
\nonumber\\
&&\times
\sum_{k=0}^{n-(i_1+\cdots+i_l)}v_{n,k+l}B_{n-(i_1+\cdots+i_l),k}(w_\bullet)
+o(1), \qquad n,r\to\infty,\qquad r\asymp n,
\label{L1_dist_G_On}
\end{eqnarray}
if $l\eta_1+\eta_2(k)>0$ and $\eta_2(k)+\eta_3(k)>0$ hold for 
$1\le l\le \lfloor n/r\rfloor$, $1\le k\le \lceil n/r\rceil-1$, and 
$\breve{w}(1)<\infty$.
\end{lem}

\subsection{The consistent Gibbs partition}

Some explicit results are available for the consistent Gibbs partition, which
is a class of Gibbs partitions whose $w$-weights have a form of 
(\ref{Gibbs_cons}). Let us begin with seeing asymptotic behavior of the 
smallest size conditioned by the number of blocks in the consistent Gibbs 
partition. For the case that the size is extremely large, $O(n)$, asymptotic 
forms of the associated partial Bell polynomials, $B_{n,k,(r)}(w_\bullet)$ with
$w$-weights (\ref{Gibbs_cons}), in $n,r\to\infty$ with $r\asymp n$ and fixed 
$k$ yield the distribution of the smallest size conditioned by the number of 
blocks immediately. The asymptotic forms are developed in Appendix B.

\begin{pro}\label{pro:LK_dist_CG_cond}
In a consistent Gibbs partition, which has the form of $(\ref{Gibbs})$ with
$w$-weights satisfies $(\ref{Gibbs_cons})$, the smallest size conditioned by 
the fixed number of blocks satisfies
\begin{equation*}
{\mathbb P}(|A_{(|\Pi_n|)}|\ge r||\Pi_n|=k)\sim\tilde{\omega}_\alpha(x,k),
\qquad n,r\to\infty,\qquad r\sim xn,
\end{equation*}
where if $x^{-1}\ge k$ then
\begin{eqnarray*}
\tilde{\omega}_\alpha(x,k)&=&\frac{\Gamma(-\alpha)}{\Gamma(-k\alpha)}
\frac{(-1)^{k-1}}{k}
{\mathcal I}^{(k-1)}_{x,x}(-\alpha;-\alpha)n^{-(k-1)\alpha},\qquad 
0<\alpha<1,\\
\tilde{\omega}_0(x,k)&=&
\frac{(\log n)^{1-k}}{k}{\mathcal I}^{(k-1)}_{x,x}(0;0),\qquad \alpha=0,\\
\tilde{\omega}_{\alpha}(x,k)&=&
{\mathcal I}^{(k-1)}_{x,x}(-\alpha;-\alpha),\qquad \alpha<0.
\end{eqnarray*}
If $x^{-1}<k$, $\tilde{\omega}_{\alpha}(x,k)=0$ for $-\infty<\alpha<1$.
\end{pro}

To derive explicit expressions of the marginal distributions asymptotic forms 
of the $v$-weights are needed. Substituting the asymptotic forms presented in
Propositions~\ref{pro:Cnkr_asymp} and \ref{pro:Snkr_asymp} in Appendix B into
(\ref{LK_dist_G}) provides following Proposition.

\begin{pro}\label{pro:LK_dist_CG_On}
In a consistent Gibbs partition, which has the form of $(\ref{Gibbs})$ with
$w$-weights satisfying $(\ref{Gibbs_cons})$ and $v$-weights satisfying
$n!v_{n,k}=f_kO(n^{1-\eta_2(k)})$, $n\to\infty$, for fixed positive integer 
$k$, the smallest size satisfies
\begin{equation*}
{\mathbb P}(|A_{(|\Pi_n|)}|\ge r)\sim\sum_{k=1}^{\lfloor x^{-1}\rfloor}f_k
\frac{n^{-\eta_2(k)}}{k!}{\mathcal I}^{(k-1)}_{x,x}(0;0), \qquad n,r\to\infty,
\qquad r\sim xn, \qquad \alpha=0,
\end{equation*}
and
\begin{equation*}
{\mathbb P}(|A_{(|\Pi_n|)}|\ge r)\sim\sum_{k=1}^{\lfloor x^{-1}\rfloor}f_k
\frac{n^{-\eta_2(k)-k\alpha}}{(-\alpha)^k\Gamma(-k\alpha)k!}
{\mathcal I}^{(k-1)}_{x,x}(-\alpha;-\alpha), \qquad n,r\to\infty,\qquad 
r\sim xn,
\end{equation*}
for $\alpha\neq 0$.
\end{pro}

Asymptotic behavior of the largest size also admits some explicit expressions.
Recall following theorem on the number of blocks in the Ewens-Pitman partition
\cite{KorwarHollander1973,Pitman1997}, which is a member of the consistent 
Gibbs partitions. If $v$-weights of a consistent Gibbs partition are 
representable as (\ref{ratio}), it coincides with the Ewens-Pitman partition 
\cite{GnedinPitman2005}. The EPPF satisfies
\begin{equation}
{\mathbb P}(|\Pi_n|_j=m_j,1\le j\le n)=\frac{(-1)^{n}}{(-\alpha)^k}
\frac{(\theta)_{k;\alpha}}{(\theta)_n}
n!\prod_{j=1}^n\left(\begin{array}{c}\alpha\\j\end{array}\right)^{m_j}
\frac{1}{m_j!}.
\label{EP}
\end{equation}

\begin{thm}[\cite{KorwarHollander1973},\cite{Pitman1999}]\label{thm:K_asymp}
For $0<\alpha<1$ and $\theta>-\alpha$ the number of blocks in the Ewens-Pitman
partition of the form $(\ref{EP})$ satisfies 
\begin{equation*}
\frac{|\Pi_n|}{n^{\alpha}}\to S_\alpha, \qquad n\to\infty,
\end{equation*}
in almost surely and $p$-th mean for every $p>0$. Here, $S_\alpha$ has the 
probability density
\begin{equation*}
{\mathbb P}(ds)=
\frac{\Gamma(1+\theta)}{\Gamma(1+\theta/\alpha)}s^{\frac{\theta}{\alpha}}
g_{\alpha}(s)ds,
\end{equation*}
where $g_{\alpha}(s)$ is the probability density of the Mittag-Leffler
distribution \cite{Pitman2006}. For $\alpha=0$ and $\theta>0$,
\begin{equation*}
\frac{|\Pi_n|}{\log n}\to \theta,\qquad a.s., \qquad n\to\infty.
\end{equation*}
For $\alpha<0$ and $\theta=-m\alpha$, $m=1,2,...$, $|\Pi|_n=m$ for all 
sufficiently large $n$ almost surely.
\end{thm}

\noindent
Theorem~\ref{thm:K_asymp} provides us idea how the number of blocks should be
scaled with $n$ to see proper conditional distribution of the largest size 
given the number of blocks in general consistent Gibbs partitions. In fact, 
Proposition~\ref{pro:GA_byG} and the asymptotic forms of the partial Bell 
polynomials with $w$-weights (\ref{Gibbs_cons}) given in Appendix B yield 
following results.

\begin{pro}\label{pro:L1_dist_CG_cond}
In a consistent Gibbs partition, which has the form of $(\ref{Gibbs})$ with
$w$-weights satisfying $(\ref{Gibbs_cons})$ and $\alpha<0$, the largest size 
conditioned by the number of blocks satisfies
\begin{equation*}
{\mathbb P}(|A_{(1)}|\le r||\Pi_n|=k)
\sim\tilde{\rho}_{\alpha}(x,k), \qquad n,r\to\infty, \qquad r\sim xn,
\end{equation*}
for fixed positive integer $k$, where if $x^{-1}\le k$ then
\begin{equation*}
\tilde{\rho}_{\alpha}(x,k)=\sum_{0\le j<x^{-1}}(-1)^j
\left(\begin{array}{c}k\\j\end{array}\right)
{\mathcal I}^{(j)}_{x,0}(-\alpha;(j-k)\alpha),
\end{equation*}
and if $x^{-1}>k$, $\tilde{\rho}_{\alpha}(x,k)=0$.
\end{pro}

\begin{rem}\rm
It may be natural to ask similar expressions for the case that $0<\alpha<1$
with $k=O(n^\alpha)$, but the author do not know such expressions. Substituting
(\ref{G^r_byG}) into (\ref{L1_dist_microG}) yields
\begin{eqnarray}
&&{\mathbb P}(|A_{(1)}|\le r||\Pi_n|=k)=\nonumber\\
&&1+\sum_{l=1}^{\lfloor (n-k)/r\rfloor}\frac{\alpha^l}{l!}
\sum_{\substack{i_1,...,i_l\ge r+1,\\i_1+\cdots+i_l\le n-k+l}}
\frac{C(n-(i_1+\cdots+i_l),k-l;\alpha)}{C(n,k;\alpha)(\Gamma(1-\alpha))^l}
(-1)^{i_1+\cdots+i_l}
[n]_{i_1+\cdots+i_l}\prod_{j=1}^l\frac{\Gamma(i_j-\alpha)}{\Gamma(i_j+1)}.
\label{L1_dist_CG_cond_eq1}
\end{eqnarray}
The asymptotic form (\ref{Cnk_asympP}) in Appendix B of the generalized 
factorial coefficients for $n\to\infty$, $k\sim sn^\alpha$, and fixed $l$ 
yields
\begin{equation}
\frac{C(n-(i_1+\cdots+i_l),k-l;\alpha)}{C(n,k;\alpha)}(-1)^{i_1+\cdots+i_l}
[n]_{i_1+\cdots+i_l}\sim(1-k)_l
\left(1-\frac{i_1+\cdots+i_l}{n}\right)^{-1-\alpha},
\label{L1_dist_CG_cond_eq2}
\end{equation}
as long as $n-(i_1+\cdots+i_l)\asymp n$. Substituting 
(\ref{L1_dist_CG_cond_eq2}) into (\ref{L1_dist_CG_cond_eq1}) yields an 
expression
\begin{eqnarray*}
\tilde{\rho}_\alpha(x,k)\sim\sum_{0\le l<x^{-1}}
\frac{\Gamma(-\alpha)}{\Gamma(-(l+1)\alpha)}\frac{s^l}{l!}
{\mathcal I}_{x,0}^{(l)}(-\alpha;-\alpha), \qquad n,r\to\infty,\qquad 
r\sim xn,\, k\sim sn^\alpha,
\end{eqnarray*}
but the incomplete Dirichlet integrals are divergent. This is because
(\ref{L1_dist_CG_cond_eq2}) does not hold in the whole domain of the summation
in (\ref{L1_dist_CG_cond_eq1}). We have similar observation for $\alpha=0$
with $k=O(\log n)$, where (\ref{Snk_asympH}) in Appendix B is employed.
\end{rem}

\begin{rem}\rm
It is worth mentioned that these expressions give asymptotic forms of the 
associated partial Bell polynomials, $B^{(r)}_{n,k}((1-\alpha)_{\bullet-1})$, 
with $n,r\to\infty$, $r\asymp n$, and appropriately scaled $k$.
\end{rem}

\subsection{The Ewens-Pitman partition}

If $v$-weights are specified in a consistent Gibbs partition further explicit
results are available. In this subsection the $v$-weights of the form 
(\ref{v_EP}), or the Ewens-Pitman partition, is discussed. The Ewens-Pitman 
partition has nice properties and appears in various contexts (see, for 
example, \cite{TavareEwens1997,Arratia2003,Pitman2006}).

\subsubsection{The smallest size}

In the Ewens-Pitman partition with $\alpha=0$ and $\theta>0$ multiplicities
of the small components are asymptotically independent (\ref{joint_Ewens}) and
this property immediately leads following theorem on asymptotic behavior of 
the smallest size in the Ewens-Pitman partition \cite{ArratiaTavare1992}. But 
evaluating (\ref{LK_dist_G_on}) also provides the theorem. We omit the proof 
because it is similar to the proof of Theorem~\ref{thm:LK_dist_P_on}.

\begin{thm}[\cite{ArratiaTavare1992}]\label{thm:LK_dist_E_on}
In the Ewens-Pitman partition of the form $(\ref{EP})$ with $\alpha=0$ and
$\theta>0$ the smallest size satisfies
\begin{equation*}
{\mathbb P}(|A_{(|\Pi_n|)}|\ge r)
\sim e^{-\theta h_{r-1}}, \qquad n\to\infty,\qquad r=o(n),
\end{equation*}
where $h_r=\sum_{k=1}^rk^{-1}$ with a convention $h_0=0$. Moreover, 
${\mathbb P}(|A_{(|\Pi_n|)}|\ge r)\sim r^{-\theta}e^{-\gamma\theta}$,  
$n,r\to\infty$, $r=o(n)$, where $\gamma$ is the Euler-Mascheroni constant.
\end{thm}

\noindent
A necessary condition of the asymptotic independence (\ref{joint_Ewens}) is 
the conditioning relation (\ref{cond_rel}) with the 
{\it logarithmic condition}:
\begin{equation*}
j{\mathbb P}(Z_j=1)\to\theta, \qquad j{\mathbb E}(Z_j)\to\theta, \qquad
j\to\infty,
\end{equation*}
for some $\theta>0$ \cite{Arratia2003}. In the Ewens-Pitman partition
the logarithmic condition holds only if $\alpha=0$. Therefore application of
the asymptotic independence (\ref{joint_Ewens}) is restricted to the case that 
$\alpha=0$. On the other hand, evaluating (\ref{LK_dist_G_on}) is possible for
non-zero $\alpha$ and gives the following theorem, whose proof is in 
Subsection~\ref{subsec:proof_P}.

\begin{thm}\label{thm:LK_dist_P_on}
In the Ewens-Pitman partition of the form $(\ref{EP})$ with $0<\alpha<1$ and
$\theta>-\alpha$ the smallest size satisfies 
\begin{equation*}
{\mathbb P}(|A_{(|\Pi_n|)}|\ge r)\sim\frac{\Gamma(1+\theta)}{\Gamma(1-\alpha)}
\left\{\sum_{j=1}^{r-1}p_{\alpha}(j)\right\}^{-1-\frac{\theta}{\alpha}}
n^{-\theta-\alpha},\qquad n\to\infty,
\end{equation*}
for $r=o(n)$, $r\ge 2$, where
\begin{equation}
p_{\alpha}(j)=\left(\begin{array}{c}\alpha\\j\end{array}\right)(-1)^{j+1},
\qquad j=1,2,... \label{Sibuya_dist}
\end{equation}
Moreover,
\begin{equation*}
{\mathbb P}(|A_{(|\Pi_n|)}|\ge r)\sim\frac{\Gamma(1+\theta)}{\Gamma(1-\alpha)}
n^{-\theta-\alpha},\qquad n,r\to\infty,\qquad r=o(n).
\end{equation*}
\end{thm}

\begin{rem}\rm
In the case $0<\alpha<1$ (\ref{Sibuya_dist}) is a probability mass function. 
Devroye called it Sibuya's distribution \cite{Sibuya1979,Devroye1993}.
\end{rem}

\begin{rem}\label{rem:LK=1}\rm
Theorem~\ref{thm:LK_dist_P_on} gives a convergence result:
\begin{equation}
|A_{(|\Pi_n|)}|=1+O_p(n^{-\theta-\alpha}), \qquad n\to\infty.
\label{LK=1}
\end{equation}
In addition, since
\begin{equation*}
{\mathbb P}(|A_{(|\Pi_n|)}|=n)\sim\frac{\Gamma(1+\theta)}{\Gamma(1-\alpha)}
n^{-\theta-\alpha},\qquad n\to\infty,
\end{equation*}
we have ${\mathbb P}(r\le |A_{(|\Pi_n|)}|<n)=o(n^{-\theta-\alpha})$ as 
$n,r\to\infty$, $r\asymp n$. Therefore apart from the mass at 
$|A_{(|\Pi_n|)}|=n$ probability mass concentrates around one. Study of 
asymptotic behavior of small sizes in an infinite exchangeable random partition
goes back to works by Karlin
\cite{Karlin1967} and Rouault \cite{Rouault1978}. See also 
\cite{YamatoSibuya2000,Pitman2006}. For an infinite exchangeable random 
partition $T_\alpha:=\lim_{n\to\infty}|\Pi_n|/n^\alpha$, where $0<\alpha<1$, is
an almost sure and strictly positive limit if and only if the ranked 
frequencies, $P_{(j)}$, $j=1,2,...$ satisfies $P_{(j)}\sim Zj^{-1/\alpha}$ as 
$j\to\infty$ with $0<Z<\infty$. In that case 
$Z^{-\alpha}=\Gamma(1-\alpha)T_\alpha$ and 
$|\Pi_n|_j\sim p_{\alpha}(j)T_\alpha n^\alpha$ as $n\to\infty$ for each $j$.
The Poisson-Dirichlet process and the Ewens-Pitman partition with $0<\alpha<1$
and $\theta>-\alpha$ satisfy these conditions and $T_\alpha$ has the 
probability density of $S_\alpha$ defined in Theorem~\ref{thm:K_asymp}. By 
noting that $\{|\Pi_n|_1=0\}=\{|A_{(|\Pi_n|)}|>1\}$ it can be seen that 
$|\Pi_n|_1\sim\alpha S_\alpha n^\alpha$, which is consistent with (\ref{LK=1}).
\end{rem}

The next proposition, whose proof is given in Subsection~\ref{subsec:proof_P},
implies that for $\alpha<0$ and $\theta=-m\alpha$, $m=1,2,...$, the smallest 
size is $\Omega(n)$ in probability.

\begin{pro}\label{pro:LK_dist_Pn_on}
In the Ewens-Pitman random partition of the form $(\ref{EP})$ with $\alpha<0$
and $\theta=-m\alpha$, $m=2,3,...$ the smallest size satisfies
\begin{equation*}
{\mathbb P}(|A_{(|\Pi_n|)}|\ge r)
\sim 1+\frac{m\Gamma(-m\alpha)}{\Gamma((1-m)\alpha)}
\left\{\sum_{j=1}^{r-1}p_{\alpha}(j)\right\}n^\alpha,\qquad n\to\infty,
\end{equation*}
for $r=o(n)$, $r\ge 2$, where $p_\alpha(\bullet)$ is defined by
$(\ref{Sibuya_dist})$. Moreover, 
${\mathbb P}(|A_{(|\Pi_n|)}|\ge r)\sim 1-O((n/r)^{\alpha})$, $n,r\to\infty$, 
$r=o(n)$.
\end{pro}

Let us see a large deviation, where the smallest size is extremely large. By 
using the conditioning relation (\ref{cond_rel}) Arratia, Barbour and Tavar\'e
\cite{Arratia2003} established following assertion for the Ewens-Pitman 
partition with $\alpha=0$ and $\theta>0$. But the assertion is the direct
consequence of Proposition~\ref{pro:LK_dist_CG_On}.

\begin{cor}[\cite{Arratia2003}]\label{cor:LK_dist_E_On}
In the Ewens-Pitman partition of the form $(\ref{EP})$ with $\alpha=0$ and 
$\theta>0$ the smallest size satisfies
\begin{equation*}
{\mathbb P}(|A_{(|\Pi_n|)}|\ge r)
\sim\Gamma(\theta)(xn)^{-\theta}\omega_{\theta}(x),\qquad n,r\to\infty,
\qquad r\sim xn,
\end{equation*}
where
\begin{equation*}
\omega_{\theta}(x)=x^\theta
\sum_{k=1}^{\lfloor x^{-1}\rfloor}\frac{\theta^k}{k!}
{\mathcal I}^{(k-1)}_{x,x}(0;0).
\end{equation*}
\end{cor}

In addition, Proposition~\ref{pro:LK_dist_CG_On} immediately gives asymptotic 
behavior of the smallest size in the Ewens-Pitman partition for non-zero 
$\alpha$.

\begin{cor}\label{cor:LK_dist_P_On}
In the Ewens-Pitman partition of the form $(\ref{EP})$ with $0<\alpha<1$ and 
$\theta>-\alpha$ the smallest size satisfies asymptotically
\begin{equation*}
{\mathbb P}(|A_{(|\Pi_n|)}|\ge r)\sim\frac{\Gamma(1+\theta)}{\Gamma(1-\alpha)}
n^{-\theta-\alpha},\qquad n,r\to\infty,\qquad r\sim xn.
\end{equation*}
For $\alpha<0$ and $\theta=-m\alpha$, $m=1,2,...$,
\begin{equation*}
{\mathbb P}(|A_{(|\Pi_n|)}|\ge r)
\sim{\mathcal I}_{x,x}^{(m-1)}(-\alpha;-\alpha),\qquad x^{-1} \ge m,
\end{equation*}
and ${\mathbb P}(|A_{(|\Pi_n|)}|\ge r)=O(n^{(m-\lfloor x^{-1}\rfloor)\alpha})$
for $x^{-1}<m$.
\end{cor}

\begin{rem}\label{Buchstab}\rm
In Corollary~\ref{cor:LK_dist_E_On} the function $\omega_1(x^{-1})$ is known as
Buchstab's function for the frequency of rough numbers in the number theory 
\cite{Buchstab1937,Tenenbaum1995}. Corollary~\ref{cor:LK_dist_P_On} shows that
we do not have ``generalized Buchstab's function'' for $0<\alpha<1$ and 
$\theta>-\alpha$.
\end{rem}

The factorial moments of the smallest size satisfy
\begin{equation*}
{\mathbb E}([|A_{(|\Pi_n|)}|]_i)=i\sum_{j=i}^n[j-1]_{i-1}
{\mathbb P}(|A_{(|\Pi_n|)}|\ge j), \qquad i=1,2,...
\end{equation*}
Theorem~\ref{thm:LK_dist_P_on}, Corollaries~\ref{cor:LK_dist_E_On} and
\ref{cor:LK_dist_P_On} provide an implication on the factorial moments of 
the smallest size, which might be important in statistical applications. The 
proof is given in Subsection~\ref{subsec:proof_P}.

\begin{thm}\label{thm:LK_mom}
In the Ewens-Pitman partition of the form $(\ref{EP})$ with $\alpha\ge 0$ 
and $\theta>-\alpha$ the $i$-th factorial moments of the smallest size exist if
and only if $\theta\ge i-\alpha$, $i=1,2,...$ $($for $\alpha=0$, $\theta>i)$.
Moreover, if $\alpha>0$ and $\theta>i-\alpha$,
\begin{equation*}
{\mathbb E}([|A_{(|\Pi_n|)}|]_i)\sim\delta_{i,1},\qquad i=1,2,...
\end{equation*}
For $\alpha<0$ and $\theta=m\alpha$, $m=1,2,...$,
\begin{equation*}
{\mathbb E}\left(\frac{[|A_{(|\Pi_n|)}|]_i}{n^i}\right)\sim
i\int_0^{m^{-1}}x^{i-1}I^{(m-1)}_{x,x}(-\alpha;-\alpha)dx, \qquad n\to\infty,
\qquad i=1,2,...
\end{equation*}
\end{thm}

\begin{rem}\rm
In applications it might be interesting to consider a test of the hypothesis 
``$\alpha>0$'', since properties of the Ewens-Pitman partition crucially depend
on sign of $\alpha$. For example, for $\alpha\le 0$ the partition is a 
one-parameter family, but it is not for $\alpha>0$. According to 
Theorems~\ref{thm:LK_dist_E_on} and \ref{thm:LK_dist_P_on} if $\alpha>0$ and 
$\theta>-\alpha$, $|A_{(|\Pi_n|)}|\overset{p}{\to}1$, $n\to\infty$, while if 
$\alpha=0$ and $\theta>0$, ${\mathbb P}(|A_{(|\Pi_n|)}|>1)\sim e^{-\theta}$. 
In addition, Proposition~\ref{pro:LK_dist_Pn_on} implies that 
${\mathbb P}(|A_{(|\Pi_n|)}|>1)\to1$, $n\to\infty$, for $\alpha<0$, 
$\theta=-m\alpha$, $m=1,2,...$ Therefore we can reject the hypothesis 
``$\alpha>0$'' by the event $\{|A_{(|\Pi_n|)}|>1\}$. This test seems powerful, 
however, might be unstable since power of the test increases with decreasing 
$\theta$, while the moments exist only for large $\theta$. For an illustration 
Table 1 displays simulation results for the number of the event 
$\{|A_{(|\Pi_n|)}|>1\}$ occurred in 10,000 trials with $n=100$.
\end{rem}

\subsubsection{The largest size}

Asymptotic behavior of the marginal distribution of the largest size in the 
Ewens-Pitman partition follows immediately from Lemma~\ref{lem:L1_dist_G_On}. 
For $0\le\alpha<1$ and $\theta>-\alpha$ the distribution is identical with the
marginal distribution of the first component of the Poisson-Dirichlet 
distribution, and the assertions of the following corollary have been 
established in studies of the Poisson-Dirichlet process (see, for example, 
\cite{Watterson1976,Griffiths1988,PitmanYor1997,ArratiaTavare1992,Arratia2003,
Handa2009}). Our proof, which is provided in Subsection~\ref{subsec:proof_P}, 
seems simpler than the treatments of the Poisson-Dirichlet process.

\begin{cor}\label{cor:L1_dist_P_On}
In the Ewens-Pitman partition of the form $(\ref{EP})$ the largest size 
satisfies
\begin{equation*}
{\mathbb P}(|A_{(1)}|\le r)\sim\rho_{\alpha,\theta}(x),\qquad n,r\to\infty,
\qquad r\sim xn,
\end{equation*}
where
\begin{eqnarray*}
\rho_{\alpha,\theta}(x)&=&\sum_{k=0}^{\lfloor x^{-1}\rfloor}
\frac{(\theta)_{k;\alpha}}{\alpha^k k!}{\mathcal I}^{(k)}_{x,0}
(-\alpha;k\alpha+\theta), \qquad 0<\alpha<1,\, \theta>-\alpha,\\
\rho_{0,\theta}(x)&=&\sum_{k=0}^{\lfloor x^{-1}\rfloor}
\frac{(-\theta)^k}{k!}{\mathcal I}^{(k)}_{x,0}(0;\theta),
\qquad \alpha=0,\, \theta>0.
\end{eqnarray*}
For $\alpha<0$ and $\theta=-m\alpha$, $m=1,2,...$,
\begin{equation*}
\rho_{\alpha,(-m\alpha)}(x)=\sum_{k=0}^{\lfloor x^{-1}\rfloor}(-1)^k
\left(\begin{array}{c}m\\k\end{array}\right)
{\mathcal I}^{(k)}_{x,0}(-\alpha;(k-m)\alpha),\qquad \lceil x^{-1}\rceil\le m
\end{equation*}
and $\rho_{\alpha,(-m\alpha)}(x)=0$, $\lceil x^{-1}\rceil>m$.
\end{cor}

\begin{rem}\label{Dickman}\rm
The function $\rho_{0,1}(x^{-1})$ is known as Dickman's function for the 
frequency of smooth numbers in the number theory 
\cite{Dickman1930,Tenenbaum1995}.
\end{rem}

Let us move to the situation that the largest size is extremely small. For 
$\alpha<0$ and $\theta=-m\alpha$, $m=1,2,...$, the largest size is $\Omega(n)$ 
almost surely, because the expression (\ref{L1_dist_Gg}) implies
\begin{equation}
{\mathbb P}(|A_{(1)}|\le r)=\sum_{k=1}^m\frac{[m]_k}{(-m\alpha)_n}
\left[\frac{\xi^n}{n!}\right]\left\{\sum_{j=0}^r
\left(\begin{array}{c}\alpha\\j\end{array}\right)(-\xi)^j\right\}^k=0,
\qquad r<\frac{n}{m}.
\label{L1_dist_Pn_On}
\end{equation}
For the case that $0<\alpha<1$ and $\theta>-\alpha$ evaluation of
(\ref{L1_dist_G_on}) leads following theorem, whose proof is given in 
Subsection~\ref{subsec:proof_P}. According to the theorem in the Ewens-Pitman 
partition with $0<\alpha<1$ the probability that the largest size is $o(n)$ 
decays exponentially. 

\begin{thm}\label{thm:L1_dist_P_on}
In the Ewens-Pitman partition of the form $(\ref{EP})$ with $0<\alpha<1$ and 
$\theta>-\alpha$ the largest size satisfies
\begin{equation*}
{\mathbb P}(|A_{(1)}|\le r)\sim\frac{\Gamma(\theta)}
{\Gamma\left(\frac{\theta}{\alpha}\right)}
\left\{-\rho_r f_r'(\rho_r)\right\}^{-\frac{\theta}{\alpha}}
\rho_r^{-n-1}
n^{\frac{\theta}{\alpha}-\theta},\qquad n\to\infty,\qquad r=o(n), 
\end{equation*}
where 
\begin{equation*}
f_r(\xi)=\sum_{j=0}^r\left(\begin{array}{c}\alpha\\j\end{array}\right)(-\xi)^j,
\end{equation*}
$f_r'(\xi)=df_r(\xi)/d\xi$ and $\rho_r$ is the unique real positive root of the
equation $f_r(\xi)=0$. Moreover,
\begin{equation*}
{\mathbb P}(|A_{(1)}|\le r)\sim\frac{\Gamma(\theta)}
{\Gamma\left(\frac{\theta}{\alpha}\right)}
\left(\frac{\alpha}{\Gamma(2-\alpha)}\right)^{-\frac{\theta}{\alpha}}
e^{-\frac{1-\alpha}{\alpha}\frac{n}{r}}
\left(\frac{n}{r}\right)^{\frac{\theta}{\alpha}-\theta},\,n,r\to\infty,
\qquad r=o(n).
\end{equation*}
\end{thm}

\noindent
Although less explicit, a similar result is available for the case that 
$\alpha=0$. The proof is given in Subsection~\ref{subsec:proof_P}.

\begin{pro}\label{pro:L1_dist_E_on}
In the Ewens-Pitman partition of the form $(\ref{EP})$ with $\alpha=0$ and 
$\theta>0$ the probability that the largest size is $o(n)$ is exponentially
small in $n$.
\end{pro}

The factorial moments of the largest size satisfy
\begin{equation*}
{\mathbb E}([|A_{(1)}|]_i)=[n]_i-i\sum_{j=i}^n[j-1]_{i-1}
{\mathbb P}(|A_{(1)}|\le j-1), \qquad i=1,2,...
\end{equation*}
Corollary~\ref{cor:L1_dist_P_On} readily gives explicit expressions of 
asymptotic forms of the factorial moments of the largest size.

\begin{cor}
In the Ewens-Pitman partition of the form $(\ref{EP})$ the largest size 
satisfies
\begin{equation*}
{\mathbb E}\left(\frac{[|A_{(1)}|]_i}{n^i}\right)\sim 
1-i\mu_{\alpha,\theta}^{(i-1)},\qquad n\to\infty,\qquad i=1,2,...,
\end{equation*}
where $\mu^{(i)}_{\alpha,\theta}$ is the $i$-th moments of 
$\rho_{\alpha,\theta}(x)$, which is defined in 
Corollary~\ref{cor:L1_dist_P_On} and an explicit expression is given in
Proposition~17 of \cite{PitmanYor1997}.
\end{cor}

\subsection{Gnedin's partition}

The partition proposed by Gnedin \cite{Gnedin2010} is a randomized version of
a symmetric Dirichlet-multinomial distribution. Gnedin's partition is 
consistent, but the $v$-weights, which are given in (\ref{v_gnedin}), are not 
representable by ratios (\ref{ratio}). In applications it is desirable to have
exchangeable random partitions of integer with finite but random number of 
blocks. The partition is a two-parameter family obtained by mixing of the 
Ewens-Pitman partition of parameter $(\alpha,\theta)=(-1,m)$ over $m$. Here, 
the Ewens-Pitman partition of parameter $(\alpha,\theta)=(-1,m)$ is the 
Dirichlet-multinomial distribution over the $m$-dimensional simplex. Gnedin 
\cite{Gnedin2010} showed that Gnedin's partition is a mixture with a mixing 
distribution of $|\Pi_\infty|:=\lim_{n\to\infty}|\Pi_n|$, where
\begin{equation}
{\mathbb P}(|\Pi_\infty|=m)=B(z_1,z_2)\frac{(s_1)_m(s_2)_m}{m!(m-1)!},
\qquad m=1,2,...
\label{numblock_gnedin}
\end{equation}
with $z_1z_2=s_1s_2=\zeta$ and $z_1+z_2=-(s_1+s_2)=\gamma$. It is possible to 
obtain asymptotic behavior of Gnedin's partition by by mixing asymptotic forms
for the Ewens-Pitman partition of  $(\alpha,\theta)=(-1,m)$, $m=1,2,...$, over
$m$, while asymptotic behavior can be addressed directly by analyzing the 
generating functions. According to Proposition~\ref{pro:LK_dist_Pn_on} and 
(\ref{L1_dist_Pn_On}) the extreme sizes in Gnedin's partition are 
asymptotically $\Omega(n)$. For the smallest size, the following corollary is a
direct consequence of Proposition~\ref{pro:LK_dist_CG_On}. 

\begin{cor} 
In Gnedin's partition, whose $w$-weights satisfy $(\ref{Gibbs_cons})$ with
$\alpha=-1$ and $v$-weights satisfy $(\ref{v_gnedin})$, the smallest size 
satisfies
\begin{equation*}
{\mathbb P}(|A_{(|\Pi_n|)}|\ge r)\sim B(z_1,z_2)
\sum_{m=1}^{\lfloor x^{-1}\rfloor}\frac{(s_1)_m(s_2)_m}{m!(m-1)!}(1-mx)^{m-1},
\qquad n,r\to\infty, \qquad r\sim xn,
\end{equation*}
where $z_1z_2=s_1s_2=\zeta$ and $z_1+z_2=-(s_1+s_2)=\gamma$. The smallest size
satisfies
\begin{equation*}
{\mathbb E}\left(\frac{[|A_{(|\Pi_n|)}|]_j}{n^j}\right)\sim
j!B(z_1,z_2)\sum_{m=1}^\infty\frac{(s_1)_m(s_2)_m}{m!(m+j-1)!m^j}=
{\mathbb E}\left\{|\Pi_\infty|^{-j}
\left(\begin{array}{c}|\Pi_\infty|+j-1\\j\end{array}\right)^{-1}\right\},
\qquad n\to\infty,
\end{equation*}
for $j=1,2,...$
\end{cor}

For the largest size, Corollary~\ref{cor:L1_dist_P_On} immediately gives 
following Corollary.

\begin{cor} 
In Gnedin's partition, whose $w$-weights satisfy $(\ref{Gibbs_cons})$ with
$\alpha=-1$ and $v$-weights satisfy $(\ref{v_gnedin})$, the largest size 
satisfies 
\begin{equation*}
{\mathbb P}(|A_{(1)}|\le r)\sim 
B(z_1,z_2)\sum_{m=\lceil x^{-1}\rceil}^\infty
\frac{(s_1)_m(s_2)_m}{m!(m-1)!}\rho_{-1,m}(x), \qquad n,r\to\infty, \qquad 
r\sim xn,
\end{equation*}
where $z_1z_2=s_1s_2=\zeta$, $z_1+z_2=-(s_1+s_2)=\gamma$, and $\rho_{-1,m}(x)$
is defined in Corollary~\ref{cor:L1_dist_P_On}. The largest size satisfies
\begin{equation*}
{\mathbb E}\left(\frac{[|A_{(1)}|]_j}{n^j}\right)\sim
1-B(z_1,z_2)j\sum_{m=2}^\infty\frac{(s_1)_m(s_2)_m}{m!(m-1)!}
\mu^{(j-1)}_{-1,m},\qquad n\to\infty,\qquad j=1,2,...
\end{equation*}
where $\mu_{-1,m}^{(j)}$ are the $j$-th moments of $\rho_{-1,m}(x)$.
\end{cor}

\section{proofs}

\subsection{Proofs for the associated partial Bell polynomials}
\label{subsec:proof_Bell}

\begin{proof}[Proof of Proposition~\ref{pro:B^{ri}nk}]
The event that the $i$-th largest size is not larger than $r$ consists of the 
disjoint events that all sizes are equal to or smaller than $r$, and the $j$ 
sizes with sum $m$ are larger than $r+1$ and remaining sizes are equal to or
smaller than $r$, where $1\le j\le i-1$. Consequently, we have
\begin{eqnarray*}
B^{(r),(i)}_{n,k}(w_\bullet)&=&B^{(r)}_{n,k}(w_\bullet)\nonumber\\
&&+\sum_{j=1}^{i-1}\sum_{\substack{m=(r+1)j\\\vee\{n-r(k-j)\}}}^{n-k+j}
\left(\begin{array}{c}n\\m\end{array}\right)B_{m,j,(r+1)}(w_\bullet)
B^{(r)}_{n-m,k-j}(w_\bullet).
\end{eqnarray*}
Summing up both hand sides of the equation in $n$ with multiplying $u^n/n!$
the first term in the right hand side gives $(\breve{w}^{(r)}(u))^k/k!$. The 
second term in the right hand side is
\begin{eqnarray*}
&&\sum_{j=1}^{i-1}\sum_{n=k}^\infty
\sum_{\substack{m=(r+1)j\\\vee\{n-r(k-j)\}}}^{n-k+j}
B_{m,j,(r+1)}(w_\bullet)\frac{u^m}{m!}
B^{(r)}_{n-m,k-j}(w_\bullet)\frac{u^{n-m}}{(n-m)!}\\
&&=\sum_{j=1}^{i-1}\sum_{m=(r+1)j}^{\infty}\sum_{l=k-j}^{r(k-j)}
B_{m,j,(r+1)}(w_\bullet)
\frac{u^m}{m!}B^{(r)}_{l,k-j}(w_\bullet)\frac{u^l}{l!}\\
&&=\sum_{j=1}^{i-1}\frac{(\breve{w}_{(r+1)}(u))^j}{j!}
\frac{(\breve{w}^{(r)}(u))^{k-j}}{(k-j)!},
\end{eqnarray*}
where the indexes are changed as $l=n-m$.
\end{proof}

\begin{proof}[Proof of Proposition~\ref{pro:GA_byG}]
The binomial expansion of the left hand side of (\ref{G^r_egf}) yields
\begin{eqnarray*}
B^{(r)}_{n,k}(w_\bullet)
&=&[u^n]\frac{n!}{k!}(\breve{w}^{(r)}(u))^k
=[u^n]\frac{n!}{k!}
\left(\breve{w}(u)-\sum_{j=r+1}^\infty w_i\frac{u^i}{i!}\right)^k
\nonumber\\
&=&B_{n,k}(w_\bullet)+n!\sum_{l=1}^{\lfloor (n-k)/r\rfloor}
\frac{(-1)^l}{l!}\sum_{\substack{i_1,...,i_l\ge r+1,\\i_1+\cdots+i_l\le n-k+l}}
[u^{n-(i_1+\cdots+i_l)}]\frac{(\breve{w}(u))^{k-l}}{(k-l)!}
\prod_{j=1}^l\frac{w_{i_j}}{i_j!}.
\end{eqnarray*}
But noting that the exponential generating function (\ref{G_egf}) gives
\begin{equation*}
[u^{n-(i_1+\cdots+i_l)}]\frac{(\breve{w}(u))^{k-l}}{(k-l)!}
=\frac{B_{n-(i_1+\cdots+i_l),k-l}(w_\bullet)}{(n-(i_1+\cdots+i_l))!}
\end{equation*}
for $n-(i_1+\cdots+i_l)\ge k-l$, we establish the expression (\ref{G^r_byG}).
The expression (\ref{Gr_byG}) can be established in the same manner.
\end{proof}

\subsection{Proof for the Gibbs partition}\label{subsec:proof_G}

\begin{proof}[Proof of Lemma~\ref{lem:L1_dist_G_On}]
By virtue of the identity (\ref{G^r_byG}), (\ref{L1_dist_G}) yields
\begin{eqnarray*}
{\mathbb P}(|A_{(1)}|\le r)&=&\sum_{k=\lceil n/r\rceil}^nv_{n,k}
B_{n,k}(w_\bullet)
+\sum_{k=\lceil n/r\rceil}^{n-r}v_{n,k}\sum_{l=1}^{\lfloor((n-k)/r\rfloor}\\
&&\times
\frac{(-1)^l}{l!}\sum_{\substack{i_1,...,i_l\ge r+1,\\i_1+\cdots+i_l\le n-k+l}}
B_{n-(i_1+\cdots+i_l),k-l}(w_\bullet)[n]_{i_1+\cdots+i_l}
\prod_{j=1}^l\frac{w_{i_j}}{i_j!}.
\end{eqnarray*}
By changing order of the summations we have
\begin{eqnarray*}
{\mathbb P}(|A_{(1)}|\le r)&=&1+
\sum_{l=1}^{\lfloor(n-\lceil n/r\rceil)/r\rfloor}
\frac{(-1)^l}{l!}
\sum_{\substack{i_1,...,i_l\ge r+1,\\i_1+\cdots+i_l\le n-\lceil n/r\rceil+l}}
[n]_{i_1+\cdots+i_l}\prod_{j=1}^l\frac{w_{i_j}}{i_j!}\\
&&\times
\sum_{k=0}^{n-(i_1+\cdots+i_l)}v_{n,k+l}B_{n-(i_1+\cdots+i_l),k}(w_\bullet)
-R_1-R_2,
\end{eqnarray*}
where $R_1:=\sum_{k=1}^{\lceil n/r\rceil-1}v_{n,k}B_{n,k}(w_\bullet)$ and
\begin{eqnarray*}
R_2:=\sum_{l=1}^{\lfloor(n-\lceil n/r\rceil)/r\rfloor}
\frac{(-1)^l}{l!}
\sum_{\substack{i_1,...,i_l\ge r+1,\\i_1+\cdots+i_l\le n-\lceil n/r\rceil+l}}
[n]_{i_1+\cdots+i_l}\prod_{j=1}^l\frac{w_{i_j}}{i_j!}
\sum_{k=0}^{\lceil n/r\rceil-l-1}v_{n,k+l}B_{n-(i_1+\cdots+i_l),k}(w_\bullet).
\end{eqnarray*}
$R_1=o(1)$ follows immediately. For $R_2$, let us consider the series
\begin{eqnarray*}
\tilde{R_2}:=
\sum_{\substack{i_1,...,i_l\ge r+1,\\i_1+\cdots+i_l\le n-\lceil n/r\rceil+l}}
[n]_{i_1+\cdots+i_l}
\prod_{j=1}^l\frac{w_{i_j}}{i_j!}\sum_{k=0}^{\lceil n/r\rceil-l-1}
v_{n,k+l}B_{n-(i_1+\cdots+i_l),k}(w_\bullet).
\end{eqnarray*}
Note that all terms are non-negative. By the assumption for the weights, we 
can take some positive real number $c$ such that
\begin{eqnarray*}
\tilde{R_2}\le c n^{-l(1+\eta_1)+1}
\sum_{\substack{i_1,...,i_l\ge r+1,\\i_1+\cdots+i_l\le n-\lceil n/r\rceil+l}}
\sum_{k=0}^{\lceil n/r\rceil-l-1}
n^{-\eta_2(k+l)}
\frac{B_{n-(i_1+\cdots+i_l),k}(w_\bullet)}{\{n-(i_1+\cdots+i_l)\}!}.
\end{eqnarray*}
The first summation can be indexed by $m=n-(i_1+\cdots+i_l)$ and the right hand
side is bounded by
\begin{eqnarray*}
&&c n^{-l(1+\eta_1)+1}
\sum_{m=\lceil n/r\rceil-l}^{n-(r+1)l}
\left(\begin{array}{c}n-m-rl-1\\l-1\end{array}\right)
\sum_{k=0}^{\lceil n/r\rceil-l-1}n^{-\eta_2(k+l)}\frac{B_{m,k}(w_\bullet)}{m!}\\
&&<c' n^{-l\eta_1}\sum_{k=0}^{\lceil n/r\rceil-l-1}n^{-\eta_2(k+l)}
\sum_{m=\lceil n/r\rceil-l}^{n-(r+1)l}\frac{B_{m,k}(w_\bullet)}{m!}
<c' n^{-l\eta_1}\sum_{k=l}^{\lceil n/r\rceil-1}n^{-\eta_2(k)}
\frac{(\breve{w}(1))^{k-
l}}{(k-l)!},
\end{eqnarray*}
where $c'$ is a positive real number. By using assumptions for $\eta_1$,
$\eta_2(k)$, and $\breve{w}(1)$, we establish $R_2=o(1)$ and the assertion 
follows.
\end{proof}

\subsection{Proofs for the Ewens-Pitman partition}\label{subsec:proof_P}

\begin{proof}[Proof of Theorem~\ref{thm:LK_dist_P_on}]
Let us evaluate (\ref{LK_dist_G_on}):
\begin{equation*}
{\mathbb P}(|A_{(|\Pi_n|)}|\ge r)=\frac{n!}{(\theta)_n}\frac{1}{2\pi\sqrt{-1}}
\oint\frac{\{h_r(\xi)\}^{-\frac{\theta}{\alpha}}}{\xi^{n+1}}d\xi,
\end{equation*}
where
\begin{equation*}
h_r(\xi):=(1-\xi)^\alpha+1-f_{r-1}(\xi),
\qquad
f_r(\xi):=\sum_{j=0}^r\left(\begin{array}{c}\alpha\\j\end{array}\right)
(-\xi)^j.
\end{equation*}
There is no root of the equation $h_r(\xi)=0$ in $|\xi|\le 1$, since
\begin{equation*}
|h_r(\xi)-1|\le\sum_{j=r}^\infty
\left(\begin{array}{c}\alpha\\j\end{array}\right)(-1)^{j-1}
\le\sum_{j=2}^\infty
\left(\begin{array}{c}\alpha\\j\end{array}\right)(-1)^{j-1}=1-\alpha<1,
\end{equation*}
for $|\xi|\le1$. The contour of the Cauchy integral is the contour introduced 
in the proof of Theorem~\ref{thm:L1_dist_P_on} with replacing the branch at 
$\xi=\rho_r$ by the branch at $\xi=1$. As does in the proof of 
Theorem~\ref{thm:L1_dist_P_on} contribution to the integral comes from the 
integral along a contour ${\mathcal H}$ with changing the variable $\xi=1+t/n$.
We have
\begin{eqnarray*}
&&\oint_{\mathcal H}
\left(1+\frac{t}{n}\right)^{-n-1}
\left\{\left(-\frac{t}{n}\right)^\alpha-\sum_{j=1}^{r-1}
\left(\begin{array}{c}\alpha\\j\end{array}\right)
\left(-1-\frac{t}{n}\right)^j\right\}^{-\frac{\theta}{\alpha}}
\frac{dt}{n}\\
&=&(1-f_{r-1}(1))^{-\frac{\theta}{\alpha}}\\
&&\times\oint_{\mathcal H}
e^{-t}\left\{1-\frac{\theta}{\alpha}(1-f_{r-1}(1))^{-1}
\left(-\frac{t}{n}\right)^\alpha+O(n^{(-1)\vee(-2\alpha)})\right\}
\frac{dt}{n}\\
&=&(1-f_{r-1}(1))^{-1-\frac{\theta}{\alpha}}\left(-\frac{\theta}{\alpha}\right)
n^{-1-\alpha}\oint_{\mathcal H}e^{-t}(-t)^\alpha dt
+O(n^{(-2)\vee(-1-2\alpha)}),
\end{eqnarray*}
where the first term of the integrand in the second line vanishes. 
Extending the rectilinear part of the contour ${\mathcal H}$ towards $+\infty$
gives a new contour ${\mathcal H}'$, and the process introduces only 
exponentially small terms in the integral. By using the Hankel representation 
of the gamma function:
\begin{equation}
\frac{1}{2\pi\sqrt{-1}}\oint_{{\mathcal H}'}
e^{-x}(-x)^{-z}dx=\frac{1}{\Gamma(z)},
\label{Hankel}
\end{equation}
the first assertion follows. The second assertion follows immediately since 
$p_\alpha(j)$, $j=1,2,...$, is a probability mass function.
\end{proof}

\begin{proof}[Proof of Proposition~\ref{pro:LK_dist_Pn_on}]
The identity (\ref{Gr_byG}) yields
\begin{eqnarray*}
&&C_r(n,k;\alpha)-C(n,k;\alpha)\\
&&=n!\sum_{l=1}^{k-1}
\frac{(-1)^l}{l!}
\sum_{\substack{1\le i_1,.,,,i_l\le r-1,\\i_1+\cdots+i_l\le n-k+l}} 
\frac{C(n-(i_1+\cdots+i_l),k-l;\alpha)}{(n-(i_1+\cdots+i_l))!}
\prod_{j=1}^l\left(\begin{array}{c}\alpha\\i_j\end{array}\right).
\end{eqnarray*}
Substituting the asymptotic form presented in Proposition~\ref{pro:Cnk_asymp} 
in Appendix B into the right hand side yields
\begin{equation*}
\frac{C_r(n,k;\alpha)-C(n,k;\alpha)}{n!}\sim
\frac{(-1)^nn^{-1-(k-1)\alpha}}{\Gamma((1-k)\alpha)(k-1)!}
\sum_{j=1}^{r-1}p_\alpha(j),\qquad n\to\infty,\qquad r=o(n).
\end{equation*}
Substituting this expression into (\ref{LK_dist_G}) and using the identity 
(\ref{C_sum}),
\begin{eqnarray*}
{\mathbb P}(|A_{(|\Pi_n|)}|\ge r)&\sim&1+\frac{m\Gamma(-m\alpha)}
{\Gamma((1-m)\alpha)}
\left\{\sum_{j=1}^{r-1}p_{\alpha}(j)\right\}n^{\alpha},\qquad n\to\infty,\qquad
r=o(n),
\end{eqnarray*}
which establishes the first assertion. 
For sufficiently large $r_0$, let
\begin{equation}
\sum_{j=1}^{r-1}p_{\alpha}(j)=\sum_{j=1}^{r_0-1}p_{\alpha}(j)
+\sum_{j=r_0}^{r-1}p_{\alpha}(j).
\label{p_sum}
\end{equation}
The first sum is bounded as
\begin{equation*}
\left|\sum_{j=1}^{r_0-1}p_{\alpha}(j)\right|
=\sum_{j=1}^{r_0-1}\prod_{k=1}^j\left(1-\frac{\alpha+1}{k}\right)
\le\sum_{j=1}^{r_0-1}\left\{\max_{k=1,...,j}\left(1-\frac{\alpha+1}{k}\right)\right\}^j,
\end{equation*}
where the maximum is less than 1 for $-1<\alpha<0$ and $(-\alpha)$ for 
$\alpha\le -1$. The second sum is
\begin{equation*}
\sum_{j=r_0}^{r-1}p_\alpha(j)=-\frac{1}{\Gamma(-\alpha)}
\sum_{j=r_0}^{r-1}\frac{\Gamma(j-\alpha)}{\Gamma(j+1)}
=-\frac{r^{-\alpha}}{\Gamma(1-\alpha)}(1+O(r_0^{-1})),
\end{equation*}
where the last equality follows by 
$\Gamma(j-\alpha)/\Gamma(j+1)\sim j^{-\alpha-1}$ for $j\ge r_0\to\infty$.
It is possible to take the limit $r>r_0\to\infty$ such that the second sum in 
(\ref{p_sum}) dominates and the second assertion follows.
\end{proof}

\begin{proof}[Proof of Theorem~\ref{thm:LK_mom}]
The assertion for $\alpha<0$ and $\theta=-m\alpha$, $m=1,2,...$ follows 
immediately from Corollary~\ref{cor:LK_dist_P_On}. Let us consider the case
that $\alpha>0$ and $\theta>-\alpha$. The expectation satisfies
\begin{equation*}
1+(n-1){\mathbb P}(|A_{(|\Pi_n|)}|=n)<{\mathbb E}(|A_{(|\Pi_n|)}|)<
1+(n-1){\mathbb P}(|A_{(|\Pi_n|)}|\ge 2).
\end{equation*}
Theorem~\ref{thm:LK_dist_P_on} implies 
${\mathbb E}(|A_{(|\Pi_n|)}|)=1+O(n^{-\theta-\alpha+1})$ as $n\to\infty$. The 
$i\,(\ge 2)$-th moments satisfy
\begin{equation*}
i\sum_{j=i}^n[j-1]_{i-1}{\mathbb P}(|A_{(|\Pi_n|)}|=n)
<{\mathbb E}([|A_{(|\Pi_n|)}|]_i)<
i\sum_{j=i}^n[j-1]_{i-1}{\mathbb P}(|A_{(|\Pi_n|)}|\ge 2).
\end{equation*}
Since $i\sum_{j=i}^n[j-1]_{i-1}=[n]_i$,
${\mathbb E}([|A_{(|\Pi_n|)}|]_i)=O(n^{-\theta-\alpha+i})$ as $n\to\infty$ and
the assertion follows. For the case that $\alpha=0$ and $\theta>0$, 
${\mathbb E}(|A_{(|\Pi_n|)}|)=\sum_{j=1}^n{\mathbb P}(|A_{(|\Pi_n|)}|\ge j)$.
Corollary~\ref{cor:LK_dist_E_On} provides an estimate
\begin{equation*}
{\mathbb E}(|A_{(|\Pi_n|)}|)=\sum_{j=1}^n{\mathbb P}(|A_{(|\Pi_n|)}|\ge j)\sim 
n^{1-\theta}\Gamma(\theta)\int_1^n u^{\theta-2}\omega_\theta(u)du,
\qquad n\to\infty.
\end{equation*}
where $u=n/j$. Since the generalized Buchstab's function satisfies 
\cite{Arratia2003}
\begin{equation*}
\frac{d}{du}\{u^\theta\omega_\theta(u)\}=\theta(u-1)^{\theta-1}\omega_\theta
(u-1),\qquad u>2,
\end{equation*}
$\omega_\theta(u)\sim c$ as $u\to\infty$ for some $c>0$. If 
$\theta\ge 1$ the integral grows as $O(n^{\theta-1})$, while if $\theta<1$ 
the integral converges. Therefore ${\mathbb E}(|A_{(|\Pi_n|)}|)<\infty$ if 
$\theta>1$. The assertion for the $i\, (\ge 2)$-th moments can be established 
in the similar manner to the argument for the case that $\alpha>0$ and 
$\theta>-\alpha$.
\end{proof}

\begin{proof}[Proof of Corollary~\ref{cor:L1_dist_P_On}]
For $0<\alpha<1$ and $\theta>-\alpha$ the assumptions of 
Lemma~\ref{lem:L1_dist_G_On} are satisfied since $\eta_1=\alpha$, 
$\eta_2(k)=\theta$, $\eta_3(k)=\alpha$ by the asymptotic form in 
Proposition~\ref{pro:Cnk_asymp} in Appendix B and $\breve{w}(1)=1/\alpha$. We
have
\begin{eqnarray*}
&&\sum_{k=0}^{n-(i_1+\cdots+i_l)}v_{n,k+l}
B_{n-(i_1+\cdots+i_l),k}(w_\bullet)\prod_{j=1}^l\frac{w_{i_j}}{i_j!}\\
&&=\frac{(-1)^n}{(-\alpha)^l}
\frac{(\theta)_{l;\alpha}}{(\theta)_n}
\prod_{i=1}^l\left(\begin{array}{c}\alpha\\i_j\end{array}\right)
\sum_{k=0}^{n-(i_1+\cdots+i_l)}\left[-\frac{\theta}{\alpha}-l\right]_k
C(n-(i_1+\cdots+i_l),k;\alpha)\\
&&=\frac{(-1)^n}{(-\alpha)^l}
\frac{(\theta)_{l;\alpha}}{(\theta)_n}
\prod_{i=1}^l\left(\begin{array}{c}\alpha\\i_j\end{array}\right)
[-\theta-\alpha l]_{n-(i_1+\cdots+i_l)},
\end{eqnarray*}
where in the last equality (\ref{C_sum}) is used. Substituting this expression
into (\ref{L1_dist_G_On}) and taking the limit $n,r\to\infty$ with 
$i_j\to y_j$, $j=1,...,l$, and $r\sim xn$, the assertion for $0<\alpha<1$ and 
$\theta>-\alpha$ follows. Then, assume $\alpha=0$ and $\theta>0$.
$\eta_1=0$, $\eta_2(k)=\theta$. For positive fixed integer $k$ the signless 
Stirling numbers of the first kind $|s(n,k)|$ satisfies asymptotically 
\cite{Jordan1947}
\begin{equation*}
\frac{|s(n,k)|}{n!}\sim\frac{1}{(k-1)!}\frac{(\log n)^{k-1}}{n},\qquad 
n\to\infty.
\end{equation*}
For $\alpha=0$ and $\theta>0$ slight modification of 
Lemma~\ref{lem:L1_dist_G_On} gives the assertion with $\eta_3(k)=0$. Finally, 
for $\alpha<0$ and $\theta=-m\alpha$, $m=1,2,...$,
$\rho_{\alpha,(-m\alpha)}(x)=0$ for $x^{-1}>m$ since the support of 
$(v_{n,k})$ is $1\le k\le m$. Since $\eta_1=\alpha$, $\eta_2(k)=\theta$, 
$\eta_3(k)=k\alpha$ by the asymptotic form in Proposition~\ref{pro:Cnk_asymp} 
in Appendix B and $\breve{w}(1)=1/\alpha$, the assumptions of 
Lemma~\ref{lem:L1_dist_G_On} are satisfied and the assertion follows.
\end{proof}

Let us prepare a lemma for the proof of Theorem~\ref{thm:L1_dist_P_on}.

\begin{lem}\label{lem:fr}
For $0<\alpha<1$ let
\begin{equation}
f_r(\xi):=\sum_{j=0}^r\left(\begin{array}{c}\alpha\\j\end{array}\right)
(-\xi)^j.
\label{fr_def}
\end{equation}
The equation $f_r(\xi)=0$ has a real positive root. Moreover, it is the unique 
root of $f_r(\xi)=0$ in $|\xi|\le \rho_r$, where $\rho_r$ is the real positive 
root satisfying
\begin{equation}
\rho_r=1+\frac{1-\alpha}{\alpha r}+O(r^{-2}),\qquad r\to\infty.
\label{fr}
\end{equation}
\end{lem}

\begin{proof}
Let us show the existence of the real positive root of the equation $f_r(x)=0$,
$x\in{\mathbb R}$. It is straightforward to see that $f_r(1)>0$ and $f_r(x)$ is
a monotonically and strictly decreasing in $x>0$, and $f_r(\infty)=-\infty$.
Hence, there exists $L>1$ such that $f_r(L)<0$. According to the intermediate 
value theorem, the real-valued continuous function $f_r(x)$, $x\in(0,L)$ there
exists the unique positive real root $\rho_r>1$ such that $f_r(\rho_r)=0$. Let
$g_r(\xi):=1-f_r(\xi)$. Because
\begin{equation*}
|g_r(\xi)|\le
\sum_{j=1}^r\left(\begin{array}{c}\alpha\\j\end{array}\right)(-1)^{j+1}|\xi|^j
<1, \qquad |\xi|<\rho_r,
\end{equation*}
$f_r(\xi)=0$ has no root in the open disk $|\xi|<\rho_r$. If 
$\rho_re^{\sqrt{-1}\phi}$, $0\le\phi<2\pi$, is another root of $f_r(\xi)=0$, 
\begin{equation*} 
\sum_{j=1}^r
\left(\begin{array}{c}\alpha\\j\end{array}\right)(-1)^{j+1}\rho_r^j
\cos(j\phi)=1
\end{equation*}
and $\phi=0$ is obvious. Therefore $\rho_r$ is the unique root of $f_r(\xi)=0$
in the closed disk $|\xi|\le\rho_r$. Since the series (\ref{fr_def}) converges
in $|\xi|\le 1$ and $f_\infty(1)=0$, let $\rho_r=1+y$, $y=o(1)$, $r\to\infty$.
By using the Taylor expansion, it can be seen that $y=-f_r(1)/f'_r(1)+O(y^2)$.
Since $f_\infty(1)=0$, similar argument to the evaluation of the second sum in
(\ref{p_sum}) provides $f_r(1)=r^{-\alpha}/\Gamma(1-\alpha)+O(r^{-\alpha-1})$.
$f'_r(1)$ is obtained in the similar manner and the assertion is established.
\end{proof}

The following proof for the first assertion is similar to the proof of 
Theorem~3.A in \cite{FlajoletOdlyzko1990}.

\begin{proof}[Proof of Theorem~\ref{thm:L1_dist_P_on}]
Let us evaluate (\ref{L1_dist_G_on}):
\begin{equation*}
{\mathbb P}(|A_{(1)}|\le r)=\frac{n!}{(\theta)_n}\frac{1}{2\pi\sqrt{-1}}
\oint\frac{\{f_r(\xi)\}^{-\frac{\theta}{\alpha}}}{\xi^{n+1}}d\xi.
\end{equation*}
Consider the Cauchy integral takes a contour (see Figure 1) 
${\mathcal C}=\gamma_1\cup\gamma_2\cup\gamma_3\cup\gamma_4$, where
\begin{eqnarray*}
\gamma_1&=&\left\{\xi=\rho_r-\frac{t}{n};\,t=e^{\sqrt{-1}\theta},\,
\theta\in\left[\frac{\pi}{2},-\frac{\pi}{2}\right]\right\},\\
\gamma_2&=&\left\{\xi=\rho_r+\frac{\eta t+\sqrt{-1}}{n};\,t\in[0,n]\right\},\\
\gamma_3&=&\left\{\xi;|\xi|=\sqrt{(\rho_r+\eta)^2+\frac{1}{n^2}};\,
\Re(\xi)\le\rho_r+\eta\right\},\\
\gamma_4&=&\left\{\xi=\rho_r+\frac{\eta t-\sqrt{-1}}{n};\,t\in[n,0]\right\}.
\end{eqnarray*}
According to Lemma~\ref{lem:fr} we can take $\eta>0$ such that no root of 
$f_r(\xi)=0$ exist in the closed disk $|\xi|\le\rho_r+\eta$ except $\rho_r$. 
The integrand is holomorphic in the disk with the single singularity at the 
origin with the cut along the real line $[\rho_r,\infty)$. The contribution of 
$\gamma_3$, which is $O((\rho_r+\eta)^{-n})$ with $\rho_r+\eta>1$, is 
exponentially small. Changing the variable $\xi=\rho_r+t/n$ and letting 
${\mathcal H}$ be the contour on which $t$ varies when $u$ varies on the rest 
of the contour, $\gamma_4\cup\gamma_1\cup\gamma_2$, yields
\begin{eqnarray*}
&&\int_{\mathcal H}\left(\rho_r+t/n\right)^{-n-1}
\left\{f_r(\rho_r+t/n)\right\}^{-\frac{\theta}{\alpha}}\frac{dt}{n}\\
&&=\{-\rho_rf'_r(\rho_r))\}^{-\frac{\theta}{\alpha}}
\rho_r^{-n-1}n^{\frac{\theta}{\alpha}-1}\int_{\mathcal H}e^{-\frac{t}{\rho_r}}
\left(-\frac{t}{\rho_r}\right)^{-\frac{\theta}{\alpha}}dt
+O(\rho_r^{-n-1}n^{\frac{\theta}{\alpha}-2}).
\end{eqnarray*}
By using the Hankel representation of the gamma function (\ref{Hankel})
the first assertion is established. For the second assertion, let us evaluate
\begin{equation*}
-\rho_r f_r'(\rho_r)
=\alpha\rho_r\sum_{j=0}^{r-1}
\left(\begin{array}{c}\alpha-1\\j\end{array}\right)(-\rho_r)^j.
\end{equation*}
For sufficiently large $r_0$, it can be seen that
\begin{equation*}
\sum_{j=r_0}^{r-1}
\left(\begin{array}{c}\alpha-1\\j\end{array}\right)(-\rho_r)^j
=\frac{r^{1-\alpha}}{\Gamma(2-\alpha)}+O(r_0^{1-\alpha}),\qquad
\sum_{j=0}^{r_0-1}
\left(\begin{array}{c}\alpha-1\\j\end{array}\right)(-\rho_r)^j
<\sum_{j=0}^{r_0-1}\rho_r^j.
\end{equation*}
Substituting (\ref{fr}) and taking the limit $r,r_0\to\infty$ with keeping
$r_0=o(r^{1-\alpha})$ the first sum dominates and the second assertion follows.
\end{proof}

\begin{proof}[Proof of Proposition~\ref{pro:L1_dist_E_on}]
Let us evaluate (\ref{L1_dist_G_on}):
\begin{equation*}
{\mathbb P}(|A_{(1)}|\le r)
=\frac{n!}{(\theta)_n}\frac{1}{2\pi\sqrt{-1}}\oint e^{(n+1)f_{r,n}(\xi)}d\xi,
\end{equation*}
where
\begin{equation*}
f_{r,n}(\xi):=\frac{\theta}{n+1}\sum_{j=1}^r\frac{\xi^j}{j}-\log\xi.
\end{equation*}
The saddle points of $f_{r,n}(\xi)$ are
\begin{equation*}
\rho_{r,n,j}=\left(\frac{n}{\theta}\right)^{\frac{1}{r}}e^{2\pi\sqrt{-1}j/r}
-\frac{1}{r}+O(n^{-\frac{1}{r}})=:\rho_je^{\sqrt{-1}\varphi_j}, \qquad 
j=0,1,...,r-1.
\end{equation*}
Taylor's expansions of $f_{r,n}(\xi)$ around the saddle points yields
\begin{equation*}
f_{r,n}(\rho_{r,n,j}+\xi_j e^{\sqrt{-1}\eta_j})
=f_{r,n}(\rho_{r,n,j})+\frac{1}{2}\left(\frac{\xi_j}{\rho_j}\right)^2
\left[1+O(n^{-\frac{2}{r}})\right]e^{2\sqrt{-1}(\eta_j-\varphi_j)}
+O\left(\frac{\xi_j}{\rho_j}\right)^3, \qquad j=0,1,...,r-1,
\end{equation*}
and thus the direction of the steepest descent of the $i$-th saddle point is 
$\eta_j=\varphi_j+\pi/2$. The contour can be deformed such that it goes 
through each saddle point along the direction of the steepest descent without 
changing the value of the Cauchy integral. The value is evaluated as
\begin{equation*}
\frac{1}{2\pi\sqrt{-1}}\oint e^{(n+1)f_{r,n}(\xi)}d\xi
\sim\frac{1}{\sqrt{2\pi n}}\sum_{j=0}^{r-1}\rho_{r,n,j}^{-n}
\exp\left(\theta\sum_{k=1}^r\frac{\rho_{r,n,j}^k}{k}+\sqrt{-1}\varphi_j\right),
\qquad n\to\infty,
\end{equation*}
and the assertion is established.
\end{proof}

\section*{Acknowledgements}
The author thanks Akinobu Shimizu for comments in connection with Section 2 and
Hsien-Kuei Hwang, Masaaki Sibuya, and Hajime Yamato for discussions and 
comments in connection with Section 4.

\appendix

\section{}

%
%

We provide some recurrence relations for the associated partial Bell 
polynomials introduced in this paper.

\begin{pro}
The associated partial Bell polynomials, $B_{n,k(r)}(w_\bullet)$, for fixed 
positive integer $r$, satisfy the recurrence relation
\begin{equation*}
B_{n+1,k,(r)}(w_\bullet)
=\sum_{j=r-1}^{n-r(k-1)}\left(\begin{array}{c}n\\j\end{array}\right)
w_{j+1}B_{n-j,k-1,(r)}(w_\bullet),
\end{equation*}
for $n=rk-1,rk,...$, $k=1,2,...$, $B_{0,0,(r)}(w_\bullet)=1$, 
$B_{j,0,(r)}(w_\bullet)=0$, $j=1,2,...$
\end{pro}

\begin{proof}
Let
\begin{equation*}
f_{r,k}(u)=\sum_{n=rk}^\infty B_{n,k,(r)}(w_\bullet)\frac{u^n}{n!}.
\end{equation*} 
Differentiating the middle and the rightmost hand sides of (\ref{Gr_egf}) 
yields
\begin{eqnarray*}
\sum_{n=rk}^\infty B_{n,k(r)}(w_\bullet)\frac{\xi^{n-1}}{(n-1)!}
&=&\breve{B}_{k-1,(r)}(\xi,w_\bullet)
\sum_{j=r}^\infty w_j\frac{\xi^{j-1}}{(j-1)!}\\
&=&\sum_{j=r-1}^\infty\sum_{m=r(k-1)}^{\infty}\frac{w_{j+1}}{j!}
B_{m,k-1,(r)}(w_\bullet)\frac{\xi^{m+j}}{m!}\\
&=&\sum_{n=rk-1}^\infty\sum_{j=r-1}^{n-r(k-1)}
\frac{w_{j+1}}{j!}B_{n-j,k-1,(r)}(w_\bullet)\frac{\xi^n}{(n-j)!},
\end{eqnarray*}
where the indexes are changed as $m=n-j$. Equating the coefficients of 
$\xi^n/n!$ in the leftmost and the rightmost hand sides yields the recurrence 
relation.
\end{proof}

The next proposition holds in the same manner so we omit the proof.

\begin{pro}
The associated partial Bell polynomials, $B_{n,k}^{(r)}(w_\bullet)$, for fixed 
positive integer $r$, satisfy the recurrence relation
\begin{equation*}
B_{n+1,k}^{(r)}(w_\bullet)=\sum_{j=0\vee (n-rk+r)}^{(r-1)\wedge (n-k+1)}
\left(\begin{array}{c}n\\j\end{array}\right)w_{j+1}
B_{n-j,k-1}^{(r)}(w_\bullet),
\end{equation*}
for $n=k-1,...,rk-1$, $k=1,2,...$ with $B_{0,0}^{(r)}(w_\bullet)=1$, 
$B_{j,0}^{(r)}(w_\bullet)=0$, $j=1,2,...$.
\end{pro}

\begin{pro}
The associated partial Bell polynomials, $B_{n,k,(r)}(w_\bullet)$, for positive
integer $r$, satisfy the recurrence relation 
\begin{equation*}
B_{n,k,(r+1)}(w_\bullet)=\sum_{j=0}^k\frac{[n]_{rj}}{j!}
\left(-\frac{w_r}{r!}\right)^jB_{n-rj,k-j,(r)}(w_\bullet)
\end{equation*}
for $n=k,k+1,...,(r+1)k$, $k=0,1,...$, with $B_{0,0,(r)}(w_\bullet)=1$, 
$B_{j,0,(r)}(w_\bullet)=0$, $j=1,2,...$
\end{pro}

\begin{proof}
We have
\begin{equation*}
\breve{B}_{k,(r+1)}(\xi,w_\bullet)=\frac{1}{k!}
\left(\sum_{j=r}^\infty w_j\frac{\xi^j}{j!}-w_r\frac{\xi^r}{r!}\right)^k
=\sum_{j=0}^k
\left(-\frac{w_r}{r!}\right)^j\frac{\xi^{rj}}{j!}
\breve{B}_{k-j,(r)}(\xi,w_\bullet),
\end{equation*}
whose expansion into power series of $\xi$ yields
\begin{eqnarray*}
&&\sum_{n=(r+1)k}^{\infty}B_{n,k,(r+1)}(w_\bullet)\frac{\xi^n}{n!}
=\sum_{j=0}^k\sum_{m=r(k-j)}^{\infty}\left(-\frac{w_r}{r!}\right)^j
B_{m,k-j,(r)}(w_\bullet)\frac{\xi^{m+rj}}{j!m!}\\
&&=\sum_{n=rk}^{\infty}\sum_{j=0}^k\left(-\frac{w_r}{r!}\right)^j
B_{n-rj,k-j,(r)}(w_\bullet)\frac{\xi^n}{j!(n-rj)!},
\end{eqnarray*}
where the indexes are changed as $m=n-rj$. Equating the coefficients of 
$\xi^n/n!$ yields the recurrence relation.
\end{proof}

The next proposition holds in the same manner so we omit the proof.

\begin{pro}
The associated partial Bell polynomials, $B^{(r)}_{n,k}(w_\bullet)$, for 
positive integer $r$, satisfy the recurrence relation 
\begin{equation*}
B^{(r+1)}_{n,k}(w_\bullet)=\sum_{j=0\vee(n-rk)}^{\lfloor(n-k)/r\rfloor}
\frac{[n]_{(r+1)j}}{j!}
\left(\frac{w_{r+1}}{(r+1)!}\right)^jB^{(r)}_{n-j(r+1),k-j}(w_\bullet)
\end{equation*}
for $n=k,k+1,...,(r+1)k$, $k=0,1,...$, with $B^{(r)}_{0,0}(w_\bullet)=1$, 
$B^{(r)}_{j,0}(w_\bullet)=0$, $j=1,2,...$.
\end{pro}

\section{}

%
%

Asymptotic forms of the generalized factorial coefficients are given in the 
next proposition. The assertion for positive $\alpha$ appears in 
\cite{Charalambides2005} as an exercise.

\begin{pro}\label{pro:Cnk_asymp}
For non-zero $\alpha$ and fixed positive integer $k$ the generalized factorial 
coefficients, $C(n,k;\alpha)$, satisfy asymptotically
\begin{equation*}
\frac{C(n,k;\alpha)}{n!}\sim\frac{(-1)^{n+k-1}}{\Gamma(-\alpha)(k-1)!}
n^{-1-\alpha},\qquad n\to\infty,\qquad\alpha>0
\end{equation*}
and
\begin{equation*}
\frac{C(n,k;\alpha)}{n!}\sim\frac{(-1)^n}{\Gamma(-k\alpha)k!}
n^{-1-k\alpha},\qquad n\to\infty,\qquad\alpha<0.
\end{equation*}
\end{pro}

\begin{proof}
Applying the generalized binomial theorem to (\ref{C_egf}) yields
\begin{eqnarray*}
\frac{C(n,k;\alpha)}{n!}&=&\frac{1}{k!}[u^n]((1+u)^\alpha-1)^k=
\frac{1}{k!}[u^n]\sum_{j=0}^k\left(\begin{array}{c}k\\j\end{array}\right)
(1+u)^{j\alpha}(-1)^{k-j}\\
&=&\frac{1}{k!}\sum_{j=1}^k\left(\begin{array}{c}k\\j\end{array}\right)
\left(\begin{array}{c}j\alpha\\n\end{array}\right)(-1)^{k-j}
=\sum_{j=1}^k\frac{\Gamma(n-j\alpha)}{\Gamma(-j\alpha)\Gamma(n+1)}
\frac{(-1)^{k+n-j}}{j!(k-j)!}\\
&=&\frac{1}{n}\sum_{j=1}^k\frac{n^{-j\alpha}}{\Gamma(-j\alpha)}
\frac{(-1)^{k+n-j}}{j!(k-j)!}(1+O(n^{-1})),
\end{eqnarray*}
where the last equality follows by 
$\Gamma(n-j\alpha)/\Gamma(n+1)\sim n^{-j-1}$ as $n\to\infty$.
\end{proof}

A more general result is available. For positive $\alpha$ Pitman
\cite{Pitman1999} showed that
\begin{equation}
\frac{C(n,k;\alpha)}{n!}\sim\frac{(-1)^{n+k}}{(k-1)!}
\alpha g_\alpha(s)n^{-1-\alpha},\qquad n\to\infty,\qquad k\sim sn^\alpha,
\label{Cnk_asympP}
\end{equation}
where $g_\alpha(s)$ is the probability density of the Mittag-Leffler 
distribution \cite{Pitman2006}. For the signless Stirling number of the first
kind Hwang \cite{Hwang1995} showed that 
\begin{equation}
\frac{|s(n,k)|}{n!}\sim\frac{(\log n)^{k-1}}{(k-1)!n}
\left[\left\{\Gamma\left(1+\frac{k-1}{\log n}\right)\right\}^{-1}
+O\left(\frac{k}{(\log n)^2}\right)\right], \qquad n\to\infty, 
\qquad 2\le k\le s\log n,
\label{Snk_asympH}
\end{equation}
and a precise local limit theorem for $k$ around
$\log n+\gamma+1/(2n)+O(n^{-2})$ is available,
where $\gamma$ is the Euler-Mascheroni constant \cite{Louchard2010}.

Then, let us develop asymptotic forms of the associated signless Stirling 
number of the first kind, $|s_r(n,k)|$, and the associated generalized 
factorial coefficients, $C_r(n,k;\alpha)$. The author is unaware of literature
in which these asymptotics are discussed. 

\begin{pro}\label{pro:Cnkr_asymp}
For non-zero $\alpha$ and integer $k$ with $1\le k<n/r$ the $r$-associated 
generalized factorial coefficients, $C_r(n,k;\alpha)$, satisfy
\begin{equation}
\frac{C_r(n,k;\alpha)}{n!}\sim\frac{(-1)^{n}}{\Gamma(-k\alpha)k!}
{\mathcal I}^{(k-1)}_{x,x}(-\alpha;-\alpha)n^{-1-k\alpha},\qquad n,r\to\infty,
\qquad r\sim xn.
\label{Cnkr_asymp}
\end{equation}
For integer $k=n/r\ge 2$, $C_r(n,k;\alpha)/n!=O(n^{-k(1+\alpha)})$.
\end{pro}

\begin{proof}
Since the assertion is trivial for $k=1$, assume $k\ge 2$. The exponential 
generating function (\ref{Cr_egf}) yields
\begin{eqnarray*}
\frac{C_r(n,k;\alpha)}{n!}&=&
\frac{1}{k!}\sum_{\substack{i_j\ge r; j=1,...,k\\i_1+\cdots+i_k=n}}
\prod_{j=1}^k\left(\begin{array}{c}\alpha\\i_j\end{array}\right)
=\frac{1}{k!}\frac{(-1)^n}{\Gamma(-\alpha)^k}
\sum_{\substack{i_j\ge r; j=1,...,k\\i_1+\cdots+i_k=n}}
\prod_{j=1}^k\frac{\Gamma(i_j-\alpha)}{\Gamma(i_j+1)}\\
&=&\frac{n^{-k(1+\alpha)}}{k!}\frac{(-1)^n}{\Gamma(-\alpha)^k}
\sum_{\substack{i_j\ge r; j=1,...,k\\i_1+\cdots+i_k=n}}
\prod_{j=1}^k\left(\frac{i_j}{n}\right)^{-1-\alpha}(1+O(n^{-1})),
\end{eqnarray*}
where the last equality follows by 
$\Gamma(i_j-\alpha)/\Gamma(i_j+1)\sim i_j^{-1-\alpha}$ for $i_j\ge r\to\infty$.

The assertion for $k=n/r\ge 2$ follows immediately. For $1\le k<n/r$,
\begin{equation*}
\sum_{\substack{i_j\ge r; j=1,...,k\\i_1+\cdots+i_k=n}}
\prod_{j=1}^{k}\left(\frac{i_j}{n}\right)^{-1-\alpha}
\to\frac{\Gamma(-\alpha)^k}{\Gamma(-k\alpha)}
{\mathcal I}^{(k-1)}_{x,x}(-\alpha;-\alpha)n^{k-1},\qquad n,r\to\infty,\qquad
r\sim xn.
\end{equation*}
\end{proof}

For the associated signless Stirling numbers of the first kind similar 
expression is available.

\begin{pro}\label{pro:Snkr_asymp}
For integer $k$ with $2\le k<n/r$ the $r$-associated signless Stirling numbers
of the first kind, $|s_r(n,k)|$, satisfy
\begin{equation*}
\frac{|s_r(n,k)|}{n!}\sim\frac{1}{k!n}{\mathcal I}^{(k-1)}_{x,x}(0;0),
\qquad n,r\to\infty, \qquad r\sim xn.
\end{equation*}
For integer $k=n/r\ge 2$, $|s_r(n,k)|/n!=O(n^{-k})$.
\end{pro}

%
%

\newpage

%
%

\begin{table}

\begin{tabular}{crrrrrrr}
         &\multicolumn{7}{c}{$\alpha$}\\
$\theta$ &$0.9$&$0.5$&$0.1$&$0$  &$-0.1$&$-0.5$&$-1$\\
\hline
$-0.01$  &  15 &1,162&8,699&   - &    - &    - &  -\\
$0$      &  12 &1,103&8,551&   - &    - &    - &  -\\
$0.01$   &  15 &1,029&8,416&9,909&    - &    - &  -\\
$0.1$    &  14 &  785&7,358&9,042&10,000&    - &  -\\
$0.5$    &   1 &  161&4,082&5,961& 7,661&10,000&  -\\
$1$      &   0 &   43&1,003&3,610& 5,391& 9,426&10,000\\
$5$      &   0 &    0&   21&   82&   244& 3,372& 8,238\\
\hline
\end{tabular}

\vspace{5mm}

\caption{Simulation results for the number of the event $\{|A_{(|\Pi_n|)}|>1\}$ 
occurred in the Ewens-Pitman partition based on 10,000 trials with $n=100$.}

\end{table}

%
%

\begin{figure}

\includegraphics[width=8cm]{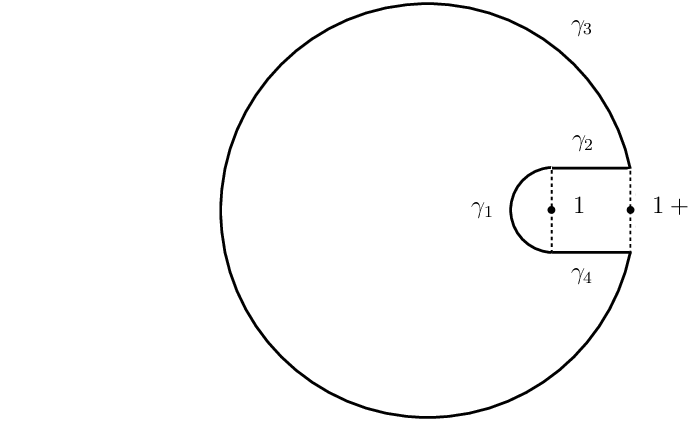}

\caption{The contour ${\mathcal C}$ used in the proof of Theorem 4.3.}

\end{figure}


\begin{thebibliography}{}

\bibitem{Aldous1985}
Aldous, D.J. (1985).
\textit{Exchangeability and related topics. Lecture Notes in Mathematics},
\textbf{1117}. Springer, Berlin.
\MR{2245368}

\bibitem{Antoniak1974}
Antoniak, C.E. (1974).
Mixtures of Dirichlet process with applications to Bayesian nonparametric 
problems.
\textit{Ann. Statist.}
\textbf{2}, 1152--1174.
\MR{0365969}

\bibitem{Aoki2002}
Aoki, M. (2002).
\textit{Modeling Aggregate Behavior and Fluctuations in Economics},
Cambridge University Press, Cambridge.
\MR{2007246}

\bibitem{Arratia2003}
Arratia, R., Barbour, A.D. and Tavar\'{e}, S. (2003).
\textit{Logarithmic Combinatorial Structures: a Probabilistic Approach}, 
European Mathematical Society, Zurich.
\MR{2032426}

\bibitem{ArratiaTavare1992}
Arratia, R. and Tavar\'{e}, S. (1992).
Limit theorem for combinatorial structures via discrete process approximations.
\textit{Random Structures Algorithms.}
\textbf{3}, 321--345. 
\MR{1164844}

\bibitem{BerestyckiPitman2007}
Berestycki, N. and Pitman, J. (2007).
Gibbs distributions for random partitions generated by a fragmentation process.
\textit{J. Stat. Phys.}
\textbf{127}, 381--418.
\MR{2314353}

\bibitem{Buchstab1937}
Buchstab, A.A. (1937).
An asymptotic estimation of a general number-theoretic function.
\textit{Mathematicheski\u{i} Sbornik}
\textbf{44}, 1239--1246.

\bibitem{Charalambides2005}
Charalambides, C.A. (2005).
\textit{Combinatorial Methods in Discrete Distributions}.
Wiley, New York.
\MR{2131068}

\bibitem{Comtet1974}
Comtet, L. (1974).
\textit{Advanced Combinatorics: The art of Finite and Infinite Expansions}.
D. Reidel, Dordrecht, Holland.
\MR{MR0460128}

\bibitem{Devroye1993}
Devroye, L. (1993).
A triptych of discrete distributions related to the stable law.
\textit{Statist. Probab. Lett.}
\textbf{18}, 349--351.
\MR{1247445}

\bibitem{Dickman1930}
Dickman, K. (1930).
On the frequency of numbers containing prime factors of a certain relative 
magnitude. 
\textit{Ark. Mat. Astr. Fys.}
\textbf{22}, 1--44.

\bibitem{Ewens1972}
Ewens, W.J. (1972).
The sampling theory of selectively neutral alleles. 
\textit{Theoret. Population Biology}
\textbf{3}, 87--112. erratum. ibid. \textbf{3} (1972), 240, 376. 
\MR{0325177}

\bibitem{Ewens1973}
Ewens, W.J. (1973).
Testing for increased mutation rate for neutral alleles.
\textit{Theoret. Population Biology}
\textbf{4}, 251--158.

\bibitem{Ferguson1973}
Ferguson, T.S. (1973).
A Bayesian analysis of some nonparametric problems.
\textit{Ann. Statist.} 
\textbf{1}, 109--230.
\MR{0350949}

\bibitem{Fisher1929}
Fisher, R.A. (1929).
Tests of significance in harmonic analysis.
\textit{Proc. Roy. Soc. Lond.}
\textbf{125}, 54--59.

\bibitem{FlajoletOdlyzko1990}
Flajolet, P. and Odlyzko, A. (1990).
Singularity analysis of generating functions.
\textit{SIAM J. Discrete Math.}
\textbf{3}, 216--240.
\MR{1039294}

\bibitem{FlajoletSedgewick2009}
Flajolet, P. and Sedgewick, R. (2009).
\textit{Analytic Combinatorics}
Cambridge University Press, New York.
\MR{2483235}

\bibitem{GnedinPitman2005}
Gnedin, A. and Pitman, J. (2005).
Exchangeable Gibbs partitions and Stirling triangles.
\textit{Zap. Nauchn. Sem. POMI}
\textbf{325}, 83--102
\MR{2160320}

\bibitem{Gnedin2010}
Gnedin, A. (2010).
A species sampling model with finitely many types.
\textit{Electron. Commun. Probab.}
\textbf{15}, 79--88.
\MR{2606505}

\bibitem{Griffiths1988}
Griffiths, R.C. (1988).
On the distribution of points in a Poisson Dirichlet process.
\textit{J. Appl. Probab.}
\textbf{25}, 336--345.
\MR{MR0938197}

\bibitem{GriffithsSpano2007}
Griffiths, R.C. and Span\`o, D. (2007).
Record indices and age-ordered frequencies in exchangeable Gibbs partition.
\textit{Electron. J. Probab.}
\textbf{12}, 1101--1130.
\MR{MR2336601}

\bibitem{Handa2009}
Handa, K. (2009).
The two-parameter Poisson-Dirichlet point process.
\textit{Bernoulli}
\textbf{15}, 1082--1116.
\MR{MR2597584}

\bibitem{Hensley1984}
Hensley, D. (1984).
The sum of $\alpha^{\Omega(n)}$ over integers $n\le x$ with all prime factors
between $\alpha$ and $y$. 
\textit{J. Number Theory}
\textbf{18}, 206--212.
\MR{0741951}

\bibitem{Hoshino2009}
Hoshino, N. (2009).
The quasi-multinomial distribution as a tool for disclosure risk assessment.
\textit{J. Official Statistics}
\textbf{25}, 269--291.

\bibitem{Hwang1995}
Hwang, H-K. (1995).
Asymptotic expansions for Stirling's number of the first kind.
\textit{J. Combin. Theory Ser.}
\textbf{A71}, 343--351.
\MR{1342456}

\bibitem{Jordan1947}
Jordan, C. (1947).
\textit{The calculus of finite differences}.
2nd. ed. Chelsea, New York.

\bibitem{Karlin1967}
Karlin, S. (1967).
Central limit theorems for certain infinite urn schemes.
\textit{J. Math. Mech.}
\textbf{17}, 373--401.
\MR{0216548}

\bibitem{Kerov1995}
Kerov, S.V. (2005).
Coherent random allocations and the Ewens-Pitman formula.
\textit{Zap. Nauchn. Sem. POMI}
\textbf{325}, 127--145.
\MR{2160323}

\bibitem{Kingman1975}
Kingman, J.F.C. (1975).
Random discrete distribution.
\textit{J. Roy. Statist. Soc. Ser.}
\textbf{B37}, 1--15.
\MR{0368264}

\bibitem{Kingman1978}
Kingman, J.F.C. (1978).
The representation of partition structures.
\textit{J. London Math. Soc.}
\textbf{18}, 374--380.
\MR{0509954}

\bibitem{Kolchin1971}
Kolchin, V.F. (1971).
A certain problem of the distribution of particles in cells, and cycles of
random permutations.
\textit{Teor. Veroyatn. Primen.}
\textbf{16}, 67--82.

\bibitem{KorwarHollander1973}
Korwar, R.M. and Hollander, M. (1973).
Contribution to the theory of Dirichlet process.
\textit{Ann. Probab.}
\textbf{1}, 705--711.
\MR{0350950}

\bibitem{Lijoi2005}
Lijoi, A., Mena R.H. and Pr\"unster, I. (2005).
Hierarchical mixture modeling with normalized inverse Gaussian priors.
\textit{J. Amer. Statist. Assoc.}
\textbf{100}, 1278--1291.
\MR{2236441}

\bibitem{Lijoi2008}
Lijoi, A., Pr\"unster, I and Walker, S.G. (2008).
Bayesian nonparametric estimators derived from conditional Gibbs structures.
\textit{Ann. Appl. Probab.}
\textbf{18}, 1519--1547.
\MR{2434179}

\bibitem{Louchard2010}
Louchard G. (2010).
Asymptotics of the Stirling number of the first kind revisited: A saddle point
approach.
\textit{Discrete Math. Theor. Comput. Sci.}
\textbf{12}, 167--184.
\MR{2676669}

\bibitem{PanarioRichmond2001}
Panario, D. and Richmond B. (2001).
Smallest components in decomposable structures: exp-log class.
\textit{Algorithmica}
\textbf{29}, 205--226.
\MR{1887304}

\bibitem{Pitman1995}
Pitman, J. (1995).
Exchangeable and partially exchangeable random partitions.
\textit{Probab. Theory Related Fields}
\textbf{102}, 145--158.
\MR{1337249}

\bibitem{Pitman1997}
Pitman, J. (1997).
Partition structures derived from Brownian motion and stable subordinators.
\textit{Bernoulli}
\textbf{3}, 79--96.
\MR{1466546}

\bibitem{Pitman1999}
Pitman, J. (1999).
Brownian motion, bridge, excursion and meander characterized by sampling at
independent uniform times. 
\textit{Electron. J. Probab.}
\textbf{4}, no 11, 33pp.
\MR{1690315}

\bibitem{Pitman2006}
Pitman, J. (2006).
\textit{Combinatorial Stochastic Processes. Lecture Notes in Mathematics},
\textbf{1875}. Springer, Berlin.
\MR{2245368}

\bibitem{PitmanYor1997}
Pitman, J. and Yor, M. (1997).
The two-parameter Poisson-Dirichlet distribution derived from a stable 
subordinator.
\textit{Ann. Probab.}
\textbf{25}, 855--900.
\MR{1434129}

\bibitem{Rouault1978}
Rouault, A. (1978).
Lois de Zipf et sources markoviennes.
\textit{Ann. Inst. H. Poincar\'e Sect. B}
\textbf{14}, 169--188.
\MR{0507732}

\bibitem{SheppLloyd1966}
Shepp, L.A. and Lloyd, S.P. (1966).
Ordered cycle length in random permutation.
\textit{Trans. Amer. Math. Soc.}
\textbf{121}, 340--357.
\MR{0195117}

\bibitem{Sibuya1979}
Sibuya, M. (1979).
Generalized hypergeometric, digamma and trigamma distributions.
\textit{Ann. Inst. Statist. Math.}
\textbf{31}, 373--390.
\MR{0574816}

\bibitem{Sobel1977}
Sobel, M., Uppuluri, V.R.R. and Frankowski K. (1977).
\textit{Selected Tables in Mathematical Statistics Vol. IV}.
American Mathematical Society. Providence, RI.
\MR{0423747}

\bibitem{TavareEwens1997}
Tavar\'{e}, S and Ewens, W.J. (1997).
``The Ewens sampling formula,'' in \textit{Multivariate Discrete 
Distributions}, N.L. Johnson, S. Kotz, N. Balakrishnan (eds.), Wiley, New York,
pp. 1--20.
\MR{1429617}

\bibitem{Tenenbaum1995}
Tenenbaum, G. (1995). 
\textit{Introduction to Analytic and Probabilistic Number Theory}.
Cambridge University Press, New York.
\MR{1342300} 

\bibitem{Watterson1976}
Watterson, G.A. (1976).
The stationary distribution of the infinite-many neutral alleles diffusion 
model.
\textit{J. Appl. Probability}
\textbf{13}, 639--651; correction. ibid. \textbf{14}, 897 (1976).
\MR{0504014}

\bibitem{WattersonGuess1977}
Watterson, G.A. and Guess, H.A. (1977).
Is the most frequent allele the oldest?
\textit{Theoret. Population Biology}
\textbf{11}, 141--160.

\bibitem{YamatoSibuya2000}
Yamato, H. and Sibuya, M. (2000).
Moments of some statistics of Pitman sampling formula.
\textit{Bull. Inform. Cybernet.}
\textbf{32}, 1--10.
\MR{1792352}

\end{thebibliography}
\end{document}